\documentclass[11pt]
{amsart}
\usepackage{amsmath,amsfonts,amssymb,amsthm}
\usepackage{amscd}
\usepackage{latexsym}
\usepackage[usenames]{color}
\usepackage{mathrsfs}
\makeindex

\usepackage[left=1in,right=1in,top=1in,bottom=1in]{geometry}

\newtheorem{deff}{Definition}[section]
\newtheorem{lemma}[deff]{Lemma}
\newtheorem{theorem}[deff]{Theorem}
\newtheorem{corollary}[deff]{Corollary}

\newtheorem{proposition}[deff]{Proposition}
\newtheorem{fact}[deff]{Fact}
\newtheorem{em-example}[deff]{Example}
\newtheorem{em-def}[deff]{Definition}        
\newtheorem{em-remark}[deff]{Remark}         
\newtheorem{em-question}[deff]{Question}

\newtheorem{claim}{Claim}

\newenvironment{example}{\begin{em-example} \em }{ \end{em-example}}
\newenvironment{definition}{\begin{em-def} \em  }{ \end{em-def}}
\newenvironment{remark}{\begin{em-remark} \em }{\end{em-remark}}
\newenvironment{question}{\begin{em-question}\em }{\end{em-question}}

\newcommand\DEQ{\hfill $\bigtriangleup$ \medskip}

\def\ker{\mathop{\rm ker}}

\def\:{\nobreak \hskip .1111em\mathpunct {}\nonscript \mkern
-\thinmuskip {:}\hskip .3333emplus.0555em\relax}

\catcode`\@=12
\def\T{{\mathbb T}}

\def\Z{{\mathbb Z}}
\def\N{{\mathbb N}}

\def\R{{\mathbb R}}
\def\Q{{\mathbb Q}}
\def\P{{\mathbb P}}

\def\J{{\mathbb Z}}

\def\cont{\mathfrak c}

\def\cont{\mathfrak{c}}

\newcommand{\IANS}{TAP}
\def\ap{$f_\omega$-productive}

\def\IANS{TAP}

\title[NSS and TAP properties in compact-like groups]{NSS and TAP properties in topological groups close to being compact}
\author[D. Dikranjan]{Dikran Dikranjan}
\author[D. Shakhmatov]{Dmitri Shakhmatov}
\author[J. Sp\v{e}v\'ak]{Jan Sp\v{e}v\'ak}

\begin{document}

\begin{abstract}
We introduce a notion of productivity (summability) of sequences in a topological group $G$, parametrized by a given function $f:\N\to\omega+1$.
The extreme case when $f$ is the function taking constant value
$\omega$ is closely related to the TAP property,   
the weaker version of the well-known property NSS.
We prove that TAP property 
 coincides with
NSS in
locally compact groups, $\omega$-bounded abelian groups
and countably compact minimal abelian groups.
As an application of our results, we provide a negative answer to \cite[Question 11.1]{Sp}.
\end{abstract}

\maketitle

Symbols $\Z$, $\Q$ and $\R$ denote topological groups of integer numbers, rational numbers and real numbers, $\T$ denotes the quotient group $\R/\Z$,
$Z(n)$ is the discrete cyclic group of order $n$. The symbol $e$ denotes the identity element of a (topological) group $G$, $\P$ denotes the set of prime numbers, $\N$ denotes the set of natural numbers, $\omega$ denotes the first infinite ordinal, and $\cont$ denotes the cardinality of the continuum.
An ordinal (in particular, a natural number) is identified with the set of all smaller ordinals.

\section{Preliminaries}

Recall that a sequence $\{a_i:i\in\N\}$ of elements of a topological
group $G$ is called {\em left Cauchy sequence} ({\em right Cauchy
sequence}) provided that for every neighborhood $U$ of the identity
of $G$ there exists $n\in\N$ such that $a_k^{-1}a_l\in U$
(respectively, $a_ka_l^{-1}\in U$) whenever $k,l$ are integers
satisfying $k\geq n$ and $l\geq n$.

\begin{definition}
We say that a topological group $G$ is {\em sequentially Weil complete} provided that every left Cauchy sequence in $G$ is convergent to some element of $G$.
\end{definition}

\begin{proposition}
Let $G$ be a topological group. The following statements are equivalent:
\begin{itemize}
\item[(i)] $G$ is sequentially Weil complete;
\item[(ii)] Every right Cauchy sequence in $G$ converges;
\item[(iii)] Every left Cauchy sequence in $G$ is also a right
Cauchy sequence  in $G$, and $G$ is sequentially closed in its Raikov completion (that is, the completion of $G$ with respect to the two-sided uniformity).
\end{itemize}
\end{proposition}

\begin{lemma}\label{G:with:no:conv:seq:is:seq:Weil:complete}
Every topological group without nontrivial convergent sequences is sequentially Weil complete.
\end{lemma}

\begin{proof}
Let $G$ be a topological group without nontrivial convergent sequences, and assume that $\{a_i:i\in\N\}$ is a left Cauchy
sequence in $G$. Then for every neighborhood $U$ of the identity $e$
of $G$ there exists $n\in\N$ such that $a_i^{-1}a_{i+1}\in U$ for every integer $i$ such that $i\ge n$. Thus, the sequence
$\{a_i^{-1}a_{i+1}:i\in\N\}$ converges to $e$. By our assumption,
there exists $n\in\N$ such that $a_i^{-1}a_{i+1}=e$ for every $i\ge n$.
Therefore, $a_i=a_n$ for every $i\ge n$, and so the sequence $\{a_i:i\in\N\}$ converges to $a_n$.
\end{proof}

\section{(Cauchy) productive sequences}

\begin{definition}
A sequence $\{b_n:n\in\N\}$ of elements of a topological group $G$ will be called:
\begin{itemize}
\item[(i)]
{\em Cauchy productive\/} provided that the sequence
$\left\{\prod_{i=0}^n b_i:n\in\N\/\right\}$ is left Cauchy;
\item[(ii)]
{\em productive\/} provided that the
sequence $\left\{\prod_{i=0}^n b_i:n\in\N\/\right\}$ converges to some element
of $G$.
\end{itemize}
\end{definition}

When the group $G$ is abelian, we will use additive notations:
\begin{definition}
A sequence $\{b_n:n\in\N\}$ of elements of an abelian topological group $G$ will be called:
\begin{itemize}
\item[(i)]
{\em Cauchy summable\/} provided that the sequence
$\left\{\sum_{i=0}^n b_i:n\in\N\/\right\}$ is a Cauchy sequence (in any of the three coinciding uniformities on $G$);
\item[(ii)]
{\em summable\/} provided that the
sequence $\left\{\sum_{i=0}^n b_i:n\in\N\/\right\}$ converges to some element
of $G$.
\end{itemize}
\end{definition}

\begin{lemma}
\label{productive:and:Cauchy:productive:coincide}
A sequence of elements of a sequentially Weil complete (abelian) group is Cauchy productive (summable) if and only if it is productive (summable).
\end{lemma}

\begin{lemma}
\label{Cauchy:productive:criterion}
For a sequence $B=\{b_n:n\in\N\}$ of elements of a topological group $G$ the following conditions are equivalent:
\begin{itemize}
\item[(i)] $B$ is Cauchy productive;
\item[(ii)] for every open neighborhood $U$ of $e$ there exists $n\in\N$ such that $\prod_{i=l+1}^m b_i\in U$ whenever $m> l\ge n$.
\end{itemize}
\end{lemma}
\begin{proof}
(i)$\to$(ii) Suppose that $B$ is Cauchy productive. Let $U$ be an
open neighborhood of $e$. Since the sequence $\{\prod_{i=0}^n
b_i:n\in\N\/\}$ is left Cauchy, there exists $n\in\N$
such that
$$
\left(\prod_{i=0}^l b_i\right)^{-1}
\left(\prod_{i=0}^m b_i\right)
=
\prod_{i=l+1}^m b_i
\in U
$$
whenever $m>l\ge n$.

(ii)$\to$(i) Let $U$ be an open neighborhood of $e$ in $G$. Without
loss of generality, we may assume that $U$ is symmetric. Let $n\in
\N$ be as in (ii), and choose $l,m\in\N$ so that $l\ge n$ and $m\ge
n$. Let $p_l=\prod_{i=0}^l b_i$ and $p_m=\prod_{i=0}^m b_i$. If
$m>l$,
 then $p_l^{-1}p_m\in U$ by (ii). If $l>m$, then
$p_l^{-1}p_m=\left(p_m^{-1} p_l\right)^{-1}\in U^{-1}=U$ by (ii).
Finally, $p_l^{-1}p_m=e\in U$ when $l=m$.
\end{proof}

\begin{corollary}
Every Cauchy productive sequence in a topological group $G$ converges to the identity $e$ of $G$.
\end{corollary}

\begin{corollary}\label{no:sequences:no:Cauchy:prod:sequences}
A topological group without nontrivial convergent sequences has no
nontrivial
Cauchy productive sequences.
\end{corollary}

\begin{example}
\label{symmetric:group:example}
 Let  $G = S(\N)$  be the symmetric group with the usual pointwise convergence topology. For each $n\in\N$, denote by $b_n$ the transposition of $n$ and $n +1$. Let $B=\{b_n:n\in\N\}\subseteq G$.
Then $G$ is a metric group with the following properties:
\begin{itemize}
\item[(i)] $G$ is two-sided (Raikov) complete;
\item[(ii)] $G$ is not Weil complete;
\item[(iii)] the sequence $B$ 
converges to the identity element $id$ of $S(\N)$;
\item[(iv)] the sequence $B$ is Cauchy productive;
\item[(v)] the sequence $B$ is not productive.
 \end{itemize}
 Indeed, (i) and (ii) are well known \cite[Chapter 7]{DPS}.
 To check (iii) it suffices to note that for every $m$ and for every $n>m$ one has $b_n(k) = k = id(k)$
 for all $k \leq m$. This proves that $b_n$ converges to the neutral element $id$ of $G$.

To prove (iv), 
let $\pi_n= b_0 b_1 \ldots b_n$. 
Given $n\in\N$ and integers $m$ and $l$ with $m>l\ge n$,
one can easily see that $b_{l+1} b_{l+2}\dots b_m$ does not move any $i\le k$,
which combined with 
Lemma \ref{Cauchy:productive:criterion}
yields that $B$ is Cauchy productive.

Let us prove (v).
Note that $\pi_n$ is the
 cycle $0\to 1\to 2\to\dots\to n\to n+1\to 0$ 
with support $\{0,1, \ldots, n, n+1\}$.
 To see that $\pi_n$ is not convergent, assume by contradiction that $\pi_n \to \sigma \in G$.
 Since $\sigma$ is surjective, there exists $k \in \N$ such that  $1 = \sigma(k)$. Then by the definition
 of the pointwise convergence topology of $G$, there exists $n_0 > k$ such that for all $n>n_0$
 one has
\begin{equation}
\label{eq:dagger}
 \pi_n(k) = \sigma(k)=1.
\end{equation}
On the other hand, the definition of $\pi_n$ implies $\pi_n(k) = k +1  \ne 1$ for $n>n_0$,  which contradicts \eqref{eq:dagger}.   Hence, $B$ is not
productive.
\end{example}

\begin{lemma}\label{sum:of:lin:is:lim:of:sum}
If $\{a_i:i\in\N\}$ and $\{b_i:i\in\N\}$ are (Cauchy) summable sequences in
an abelian topological group $G$, then the sequence $\{a_i+b_i:i\in\N\}$ is also
(Cauchy) summable.
\end{lemma}
\begin{proof}
Suppose that $\{a_i:i\in\N\}$ and $\{b_i:i\in\N\}$ are Cauchy
summable. Let $U$ be an open neighborhood of $0$ in $G$. Choose an
open neighborhood of $0$ such that $V+V\subseteq U$. By Lemma
\ref{Cauchy:productive:criterion}, we can find $n\in\N$ such that
$\sum_{i=l+1}^m a_i\in V$ and $\sum_{i=l+1}^m b_i\in V$ whenever $m$
and $l$ are integers satisfying $m>l\ge n$. Since $G$ is abelian,
$\sum_{i=l+1}^m (a_i+b_i)= \sum_{i=l+1}^m a_i + \sum_{i=l+1}^m b_i
\in V+V\subseteq U$. Applying Lemma
\ref{Cauchy:productive:criterion} once again, we conclude that the
sequence $\{a_i+b_i:i\in\N\}$ is Cauchy summable.

A similar proof for summable sequences is left to the reader.
\end{proof}

\section{$f$-(Cauchy) productive sequences, for a given function $f:\N\to\omega+1$}

\begin{definition}
$f_\omega:\N\to\omega+1$ is the function defined by
$f_\omega(n)=\omega$ for all $n\in\N$, and $f_1:\N\to\omega+1$ is
the function defined by $f_1(n)=1$ for all $n\in\N$.
\end{definition}

\begin{definition}
Let 
$f:\N\to\omega+1$ and $g:\N\to\omega+1$ be functions.
\begin{itemize}
\item[(i)] $f\le g$ means that $f(n)\le g(n)$ for every $n\in\N$. 
\item[(ii)]
$f\le^* g$ means that there exists $i\in\N$ such that $f(n)\le g(n)$
for all integers $n$ satisfying $n\ge i$.
\end{itemize}
\end{definition}

For a function $z:\N\to\Z$ we denote by $|z|$ the function $f:\N\to\N$ defined
by $f(n)=|z(n)|$ for all $n\in\N$. 

\begin{definition}\label{def:ap-seqence} Let $f:\N\to \omega+1$ be a function.
We say that a faithfully indexed sequence
$\{a_n:n\in\N\}$
 of elements of
a
topological group $G$ is:
\begin{itemize}
\item[(i)]
{\em $f$-Cauchy productive ($f$-productive) in $G$\/}
provided that
the sequence $\{a_n^{z(n)}:n\in\N\}$
is
Cauchy productive
(respectively, productive) in $G$
for every function $z:\N\to\Z$ such that 
$|z|\le f$;
\item[(ii)]
{\em $f^\star$-Cauchy productive ($f^\star$-productive) in $G$\/}
provided that
the sequence $\{a_n^{z(n)}:n\in\N\}$
is
Cauchy productive
(respectively, productive) in $G$
for every function $z:\N\to\N$ such that 
$z\le f$.
\end{itemize}
\end{definition}

When the group $G$ is abelian, we will use the natural variant of the above notions with ``productive'' replaced by ``summable''.

\begin{lemma}\label{if:f<g}
Let $f:\N\to\omega+1$ and $g:\N\to\omega+1$ be functions such that
$g\le^*f$.
Then
 every $f$-productive sequence ($f$-Cauchy productive sequence)
is also a $g$-productive sequence (a $g$-Cauchy productive sequence).
\end{lemma}

Our next proposition demonstrates that the notions of  $f^\star$-productive sequences and $f^\star$-Cauchy productive sequences may only lead to something new in non-commutative groups.

\begin{proposition}\label{fstar:is:f:for:abel} Let $f:\N\to \omega+1$ be a function. A
sequence  $\{a_n:n\in\N\}$
 of elements of an abelian topological group $G$
is $f$-summable ($f$-Cauchy summable) if and only if it is $f^\star$-summable
($f^\star$-Cauchy summable, respectively).
\end{proposition}
\begin{proof}
The ``only if'' part is obvious. To show the ``if'' part,
suppose that the sequence $\{a_n:n\in\N\}$ is $f^\star$-summable
($f^\star$-Cauchy summable).
Fix a function $z:\N\to\Z$ such that
$|z|\le f$.
Define functions $z_+:\N\to\N$ and $z_-:\N\to\Z$ by $z_+(n)=\max\{0,z(n)\}$
and $z_-(n)=\min\{0,z(n)\}$ for each $n\in\N$. Since $\{a_n:n\in\N\}$ is $f^\star$-summable, sequences
$\{z_+(i)a_i:i\in\N\}$ and $\{z_-(i)a_i:i\in\N\}$
are (Cauchy) summable.
Since
$z(i)=z_-(i)+z_+(i)$ for every $i\in\N$,
from Lemma \ref{sum:of:lin:is:lim:of:sum}
 we conclude that
the sequence $\{z(i)a_i:i\in\N\}$ is (Cauchy) summable.
Therefore, the sequence $\{a_i:i\in\N\}$ is $f$-summable (respectively, $f$-Cauchy summable).
\end{proof}

The following simple but important lemma  provides a criterion for a sequence to be $f$-Cauchy productive. It is an analogue of the Cauchy
criterion for convergence of a series in the real line.

\begin{lemma}\label{left:cauchy:condition}
Let $f:\N\to \omega+1$ be a function and $A=\{a_n:n\in\N\}$ a faithfully indexed sequence of elements of a topological group $G$.
Then the following statements are equivalent:
\begin{itemize}
\item[(i)] $A$ is 
$f$-Cauchy productive 
in $G$.
\item[(ii)]
 For every neighborhood $U$ of the identity of $G$ and every 
function $z:\N\to\Z$
with $|z|\le f$,
there exists $n\in\N$
such that $\prod_{i=l}^m a_i^{z(i)}\in U$ for every $l,m\in\N$ satisfying $n\leq l\leq m$.
\end{itemize}
\end{lemma}
\begin{proof}
This follows from Lemma \ref{Cauchy:productive:criterion} and
Definition 
\ref{def:ap-seqence}.
\end{proof}

Let us state explicitly a particular case of Lemma \ref{left:cauchy:condition}.

\begin{lemma}
\label{AP:is:null:sequence} Let $f: \N\to \omega+1$ be a function 
such that 
$f\ge f_1$.
Suppose that $\{a_n:n\in\N\}$ is an
$f$-Cauchy productive sequence in a topological group $G$.
$z:\N\to\Z$ is a function with $|z|\le f$.
Then the sequence $\{a_n^{z(n)}:n\in\N\}$ converges to the identity $e$ of $G$
for every function $z:\N\to\Z$ satisfying $|z|\le f$.
In particular, $\{a_n:n\in\N\}$ converges to $e$.
\end{lemma}

\begin{lemma}
Let $f:\N\to \omega+1$ be a function and $A=\{a_n:n\in\N\}$ be a
faithfully indexed sequence of elements of a topological group $G$.
\begin{itemize}
\item[(i)]
If $A$ is $f$-productive, then $A$ is $f$-Cauchy productive.
\item[(ii)]
If one additionally assumes that $G$ is Weil complete, then
 $A$ is an $f$-productive sequence in $G$ if and only if it is an $f$-Cauchy productive sequence in $G$.
\end{itemize}
\end{lemma}

Recall that a topological group $G$ is called {\em NSS\/}
provided that there exists a neighborhood of the identity of $G$
that contains no nontrivial subgroup of $G$.

Our next theorem strengthens \cite[Theorem 4.9]{Sp}.

\begin{theorem}\label{NSS:is:NACP} An NSS group
contains no $f_\omega$-Cauchy productive sequences.
\end{theorem}

\begin{proof}
Let $A=\{a_i:i\in\N\}$ be a
faithfully indexed
sequence in an NSS group $G$.
Fix a neighborhood $U$ of the identity $e$ of $G$ witnessing that $G$ is NSS. Then for all
$n\in\N$ such that $a_n\neq e$ there exists $z_n\in\Z$ such that $a_n^{z_n}\not\in U$.
Since
the sequence $\{a_n^{z_n}:n\in\N\}$ does not converge to $e$,
$A$ is not
$f_\omega$-Cauchy productive in $G$ by Lemma \ref{AP:is:null:sequence}.
\end{proof}

\begin{definition}
\label{f:productive:sets}
Let $f:\N\to\omega+1$ be a function. A
faithfully indexed set $A=\{a_n:n\in\N\}$ of a topological group $G$
will be called {\em an $f$-Cauchy productive set\/} (an {\em
$f$-productive set\/}) provided that the sequence
$\{a_{\varphi(n)}:n\in\N\}$ is $(f\circ\varphi)$-Cauchy productive
(respectively, $(f\circ\varphi)$-productive) for every bijection
$\varphi:\N\to\N$.
\end{definition}

\begin{example}
Define $f:\N\to \omega+1$ by $f(2n+1)=1$ and $f(2n)=\omega$ for
every $n\in\N$. For $i\in\N$ let $g_{2i}\in G=\R^\omega$ be such
that $g_{2i}(i)=1$ and $g_{2i}(j)=0$ whenever $j$ is a natural
number distinct from $i$. Further, define $g_{2i+1}\in G$ by
$g_{2i+1}(j)=0$ for $j\neq 1$ and $g_{2i+1}(1)=\frac{1}{2^i}$. Then
it is easy to check that $\{g_i:i\in\N\}$ is an $f$-productive
sequence in $G$ but it is not an $f$-productive set in $G$.
\end{example}

\begin{definition}
{\rm (\cite{Sp})} A topological group $G$ is called {\em TAP\/} (or
a {\em TAP group\/}) if $G$ contains no $f_\omega$-productive set.
\end{definition}

\section{$f$-(Cauchy) productive sets in metric groups}

\begin{lemma}\label{reshuffling}
Let $\{U_j:j\in\N\}$ be a family of subsets of a group $G$ such that
$e\in U_{j+1}^3\subset U_{j}$ for every integer $j\in\N$.
If $n\in\N$ and $\varphi:\{0,\dots,n\}\to \N\setminus\{0\}$ is an injection,
then
$\prod_{j=0}^n U_{\varphi(j)}\subseteq U_{k-1}$, where $k=\min\{\varphi(j):
j=0,\ldots,n\}$.
\end{lemma}
\begin{proof}
We will proceed by induction on $n$. If $n=0$, then $\prod_{j=0}^n
U_{\varphi(j)}=U_{\varphi(0)}=U_k\subseteq U_{k-1}$.

Assume that $n\in\N\setminus\{0\}$, and suppose that
the conclusion of our lemma holds for every integer $n'<n$.
Given an injection $\varphi:\{0,\dots,n\}\to \N\setminus\{0\}$, choose
an integer $m$ such that $0\le m\le n$ and $\varphi(m)=\min\{\varphi(j):j=0,\ldots,n\}=k$.
Applying the inductive assumption, we obtain
$$\prod_{j=0}^n
U_{\varphi(j)}=\left(\prod_{j=0}^{m-1}U_{\varphi(j)}\right)U_{\varphi(m)}\left(\prod_{j=m+1}^nU_{\varphi(j)}\right)\subseteq
U_{k}U_{k}U_{k}\subseteq U_{k-1},
$$
where the
left or the right factor of the product may
be equal to
the identity of $G$ if $m=0$ or $m=n$.
\end{proof}

Our next lemma offers a way to build $f$-Cauchy  productive sets in
a metric group.

\begin{lemma}
\label{building:f-Cauchy:sequences:in:a:metric:group} Let
$f:\N\to\omega+1\setminus\{0\}$ be a function, and let
 $\{U_n:n\in\N\}$
be a sequence of symmetric open neighborhoods forming a base at the
identity $e$ of a metric group $G$ such that $U_{n+1}^3\subseteq
U_n$ for every $n\in\N$. Assume that  $\{a_n:n\in\N\}$ is a
faithfully indexed sequence of elements of $G$ such that $\{a_n^z:
z\in\Z, |z|\le f(n)\}\subseteq U_n$ for every $n\in\N$. Then
$\{a_n:n\in\N\}$ is an $f$-Cauchy productive set.
\end{lemma}
\begin{proof}
According to Definition \ref{f:productive:sets},
we must prove that the sequence
$\{a_{\varphi(n)}:n\in\N\}$ is $(f\circ\varphi)$-Cauchy productive 
for every bijection $\varphi:\N\to\N$. Fix such a bijection $\varphi$.
Let $z:\N\to\Z$ be a function such that 
$|z|\le f\circ\varphi$.
Note that $|z(n)|\le f(\varphi(n))$ for every $n\in\N$, so 
by our assumption,
\begin{equation}
\label{b_n:equation}
b_n= a_{\varphi(n)}^{z(n)}\subseteq U_{\varphi(n)}
\mbox{ for every }
n\in\N.
\end{equation}
It remains only to show that the sequence $\{b_{n}:n\in\N\}$
is Cauchy productive.
Let
$U$ be an open neighborhood of $e$ in
$G$. There exists $k\in\N\setminus\{0\}$ such that $U_{k}\subseteq
U$. Define $n=\max\{\varphi^{-1}(j):j=0,1,\dots, k\}+1$. Suppose that
$l,m\in\N$ and  $m> l\ge n$.
Then $k'=\min\{\varphi(j):j\in\N, l\le j\le m\}>k$, and so
$$
\prod_{j=l+1}^m 
b_j
\in
U_{\varphi(l+1)}\dots U_{\varphi(m)}
\subseteq U_{k'-1}\subseteq U_k
\subseteq U
$$
by \eqref{b_n:equation} and Lemma \ref{reshuffling}.
Applying the implication (ii)$\to$(i) of Lemma \ref{Cauchy:productive:criterion},
we conclude that the sequence 
$\{b_n:n\in\N\}$ is 
Cauchy productive.
\end{proof}

\begin{theorem}\label{non1-TAP}
A non-discrete metric group contains an
$f$-Cauchy productive set for every function
$f:\N\to \omega\setminus\{0\}$.
\end{theorem}
\begin{proof}
Let $G$ be a non-discrete metric group, and let $\{U_n:n\in\N\}$ be
a base at $e$ as in the assumption of Lemma
\ref{building:f-Cauchy:sequences:in:a:metric:group}. For every
$n\in\N$, the set $V_n=\{x\in G: x^k\in U_n$ for all $k\in\Z$ with
$|k|\le f(n)\}$ is open in $G$. Since $G$ is non-discrete, by
induction on $n\in\N$ we can choose $a_n\in
V_n\setminus\{a_0,\dots,a_{n-1}\}$. Clearly, $A=\{a_n:n\in\N\}$ is
faithfully indexed. Applying Lemma
\ref{building:f-Cauchy:sequences:in:a:metric:group}, we conclude
that the set $A$ is $f$-Cauchy productive.
\end{proof}

\begin{corollary}\label{no:f-product:seq:in:nondiscrete:metr:group}
A non-discrete metric Weil complete  group contains an
$f$-productive set for every function
$f:\N\to \omega\setminus\{0\}$.
\end{corollary}

The real line $\mathbb{R}$ is NSS, so it does not contain
$f_\omega$-Cauchy productive sequences by Theorem \ref{NSS:is:NACP}.
This shows that one cannot replace  $f:\N\to\omega\setminus\{0\}$ by
$f_\omega$ either in Theorem \ref{non1-TAP}
or Corollary \ref{no:f-product:seq:in:nondiscrete:metr:group}.
Moreover,
the non-existence of $f_\omega$-Cauchy productive sequences
characterizes the NSS property in metric groups.

\begin{theorem}\label{metric:NACP:iff:NSS}
For a metric group $G$ the following conditions are equivalent:
\begin{itemize}
\item[(i)] $G$ is NSS,
\item[(ii)] $G$ does not contain an $f_\omega$-Cauchy productive sequence,
\item[(iii)] $G$ does not contain an $f_\omega$-Cauchy productive set (that is, $G$ is TAP).
\end{itemize}
\end{theorem}

\begin{proof}
The implication (i)$\to$(ii) is proved in
Theorem \ref{NSS:is:NACP}.
The implication (ii)$\to$(iii) is clear.

It remains to show that (iii)$\to$(i). Assume that $G$ is not NSS.
Let $\{U_n:n\in\N\}$ be 
as in the assumption of
Lemma~\ref{building:f-Cauchy:sequences:in:a:metric:group}. Since $G$
is not NSS, for every $n\in\N$ there exists $a_n\in G\setminus\{0\}$
such that $\{a_n^z:z\in\Z\}\subseteq U_n$. Since the sequence
$\{U_n:n\in\N\}$ is decreasing, we may assume, without loss of
generality, that $a_n\not\in \{a_0,\dots,a_{n-1}\}$ for every
$n\in\N$. Thus, $A=\{a_n:n\in\N\}$ is a faithfully indexed sequence
satisfying the assumption of Lemma
\ref{building:f-Cauchy:sequences:in:a:metric:group} with
$f=f_\omega$. Applying this lemma, we conclude that the set $A$ is
$f_\omega$-Cauchy productive.
\end{proof}

Combining Theorem \ref{metric:NACP:iff:NSS} and Lemma
\ref{productive:and:Cauchy:productive:coincide}, we obtain the following

\begin{corollary}\label{weil:NSS:iff:TAP}
For a metric Weil complete group the following conditions are equivalent:
\begin{itemize}
\item[(i)] $G$ is NSS,
\item[(ii)] $G$ does not contain an $f_\omega$-productive sequence,
\item[(iii)] $G$ does not contain an $f_\omega$-productive set (that is, $G$ is TAP).
\end{itemize}
\end{corollary}

The implication (ii)$\to$(i) of Corollary \ref{weil:NSS:iff:TAP} has been recently proved in \cite{DT}.

The assumption that $G$ is metrizable is essential in both
Theorem \ref{metric:NACP:iff:NSS} and Corollary \ref{weil:NSS:iff:TAP}; see
Example
\ref{ex:not:NSS:NACP}.

\section{$f$-(Cauchy) productive sets in linear groups}

 Call a topological group $G$  {\em  linear\/} (and its topology a {\em linear group topology\/})
 if $G$ has a base of neighborhoods of the neutral element formed by open subgroups of $G$. Clearly,
 a linear group is NSS if and only if it is discrete. Therefore,
 the non-discrete linear groups can be considered as strongly missing the NSS property.

In a presence of a linear topology Lemma \ref{left:cauchy:condition} can be simplified substantially.

\begin{theorem}\label{AP-Cauchy:seq:in:lin:groups}
For a faithfully indexed sequence $\{a_n:n\in\N\}$ of points of a linear group $G$, the following statements are equivalent:
\begin{itemize}
\item[(i)]
$\{a_n:n\in\N\}$ converges to the identity element of $G$;
\item[(ii)] $\{a_n:n\in\N\}$ is an $f$-Cauchy productive sequence for some function $f:\N\to\omega+1$
such that 
$f\ge f_1$;
\item[(iii)] $\{a_n:n\in\N\}$ is an $f$-Cauchy productive sequence for every function $f:\N\to\omega+1$
such that 
$f\ge f_1$;
\item[(iv)] $\{a_n:n\in\N\}$ is an
$f_\omega$-Cauchy productive set in $G$.
\end{itemize}
\end{theorem}

\begin{proof}
(i)$\to$(iv) 
Take an arbitrary bijection $\varphi:\N\to\N$. To establish (iv), it suffices to show that the sequence $\{a_{\varphi(n)}:n\in\N\}$ is $(f_\omega\circ\varphi)$-Cauchy productive.
Since $f_\omega=f_\omega\circ\varphi$, 
we must prove that, for an arbitrary function
$z:\N\to\N$,
the sequence $\{a_{\varphi(n)}^{z(n)}:n\in\N\}$
is Cauchy productive.
Let $U$ be a
neighborhood of the identity of $G$. Since $G$ is linear, there
exists an open subgroup $H$ of $G$ with $H\subset U$. Since
$\{a_n:n\in\N\}$ converges to $e$ by (i), so does the sequence
$\{a_{\varphi(n)}:n\in\N\}$. Therefore, there is some $n\in\N$ such
that $a_{\varphi(k)}\in H$ for every integer $k\geq n$. Since $H$ is
a group, it follows that $\prod_{j=l+1}^m
a_{\varphi(j)}^{z(j)}\in H\subset U$ for every $l,m\in\N$
such that $n\leq l< m$. Therefore, the sequence
$\left\{a_{\varphi(n)}^{z(n)}:n\in\N\right\}$ is Cauchy
productive by Lemma 
\ref{Cauchy:productive:criterion}.

(iv)$\to$(iii) and (iii)$\to$(ii) are straightforward.

The implication (ii)$\to$(i) follows from Lemma
\ref{AP:is:null:sequence}.
\end{proof}

\begin{corollary}
\label{Weil:complete:linear:group}
For a faithfully indexed sequence $\{a_n:n\in\N\}$ of points of a sequentially Weil complete
linear group $G$, the following statements are equivalent:
\begin{itemize}
\item[(i)]
$\{a_n:n\in\N\}$ converges to the identity element of $G$;
\item[(ii)] $\{a_n:n\in\N\}$ is an $f$-productive sequence for some function $f:\N\to\omega+1$
such that 
$f\ge f_1$;
\item[(iii)] $\{a_n:n\in\N\}$ is an $f$-productive sequence for every function $f:\N\to\omega+1$
such that 
$f\ge f_1$;
\item[(iv)] $\{a_n:n\in\N\}$ is an
$f_\omega$-productive set in $G$.
\end{itemize}
\end{corollary}

\begin{corollary}\label{NACP:iff:no:conv:sequences}\label{char:of:complete:linear:TAP:groups}
A linear (sequentially Weil complete) group contains an $f_\omega$-Cauchy
productive set (an $f_\omega$-
productive set) if and only if it contains a nontrivial convergent
sequence.
\end{corollary}

Theorem \ref{non1-TAP} can be significantly strengthen for linear groups.
(Note that every complete group is trivially sequentially
complete.)

\begin{corollary}
\label{sequentially:complete:liner:TAP:group:is:discrete} A
non-discrete  (sequentially Weil complete) sequential linear group
contains an $f_\omega$-Cauchy productive set (an
$f_\omega$-productive set).
\end{corollary}
\begin{proof}
A non-discrete sequential group contains a nontrivial convergent sequence, and
we can
apply Corollary \ref{char:of:complete:linear:TAP:groups}.
\end{proof}

\begin{corollary}
\label{linear:complete:corollary}
A non-discrete (Weil complete) metric linear group contains an  $f_\omega$-Cauchy productive set
(an $f_\omega$-productive set).
\end{corollary}

\begin{corollary} For a linear sequentially Weil complete
sequential group $G$ the following statements are equivalent:
\begin{itemize}
\item[(i)] $G$ is TAP.
\item[(ii)] $G$ is discrete.
\end{itemize}
\end{corollary}

\begin{theorem}
For a linear group $G$,
the following statements are equivalent:
\begin{itemize}
\item[(i)] $G$ is sequentially  Weil complete and does not contain an $f_\omega$-productive sequence,
\item[(ii)] $G$ is sequentially  Weil complete and
 does not contain an $f_\omega$-productive set,
\item[(iii)] $G$ has no nontrivial convergent sequences.
\end{itemize}
\end{theorem}
\begin{proof}
(i)$\to$(ii) holds since every $f_\omega$-productive set is an
$f_\omega$-productive sequence.

(ii)
$\Rightarrow$(iii) follows from Corollary \ref{char:of:complete:linear:TAP:groups}.

(iii)$\to$(i) follows from Lemma
\ref{G:with:no:conv:seq:is:seq:Weil:complete} and Corollary
\ref{char:of:complete:linear:TAP:groups}.
\end{proof}

Let us give an example of a topological group satisfying three equivalent conditions of the above theorem.

\begin{example}\label{ex:B:with:cocount:top}\label{ex:not:NSS:NACP}
Take the Boolean group $B$ of size $\cont$ and consider there the
co-countable topology having as a base of neighborhoods of $0$ all
subgroups of at most countable index. Let $G$ be the completion of
$B$. 
\begin{itemize}
\item[(i)]
{\em $G$ is a complete non-discrete linear group without
nontrivial convergent sequences\/} because every $G_\delta$-subset
of $G$ is open in $G$.
\item[(ii)]
{\em $G$ is not NSS and 
does not contain $f_\omega$-Cauchy productive
sequence (set)\/}. 
Indeed, $G$ is not 
not NSS because it is non-discrete and linear, and 
$G$ does not contain $f_\omega$-Cauchy productive
sequence (set) by Corollary \ref{NACP:iff:no:conv:sequences}. 
\end{itemize}
\end{example}
Example \ref{ex:not:NSS:NACP}
shows that
metrizability can not be dropped from the
assumptions of Theorem \ref{metric:NACP:iff:NSS}.

\begin{remark} 
Let  $G = S(\N)$  be the symmetric group 
discussed in Example
\ref{symmetric:group:example}.
Note that {\em $G$ has a liner metric topology\/}.
This example shows that one cannot replace ``sequentially Weil
complete'' by ``Raikov complete'' in Corollary
\ref{Weil:complete:linear:group}.
\end{remark}

\section{Descriptive properties of groups having $f_1^\star$-productive sequences}

\begin{theorem}\label{size:of:f_1-prductive:groups>c} A topological group that contains an $f_1^\star$-productive sequence has
size at least $\cont$. Furthermore, a metric group having an $f_1^\star$-productive
sequence contains a (subspace homeomorphic to the) Cantor set.
\end{theorem}

\begin{proof}
Let $A=\{a_n:n\in\N\}$ be an $f_1^\star$-productive sequence in a topological group $G$. By Lemma \ref{AP:is:null:sequence},
\begin{equation}
\label{limit:is:e}
\lim_{k\to\infty} a_k=e.
\end{equation}
If $G$ is metric, we fix a metric on $G$ satisfying $\mathrm{diam}\ G\le 1$.

By induction on $n\in\N$ we define a family $\{U_f:f\in 2^n\}$ of open subsets of $G$ and
a family $\{\mu_f: f\in 2^n\}$ of order-preserving injections from $n$ to
$\N$ satisfying the following properties:
\begin{itemize}
\item[(i$_n$)]
$U_f\cap U_g=\emptyset$ whenever $f,g\in 2^n$ and $f\not=g$,
\item[(ii$_n$)] $b_{f,n}=\prod_{i=0}^{n-1} a_{\mu_f(i)}\in U_f$ for every $f\in 2^n$,
\item[(iii$_n$)] if $f\in 2^n$ and $m< n$, then $\overline{U_f}\subseteq U_{f\restriction_m}$,
\item[(iv$_n$)] if $f\in 2^n$ and $m< n$, then $\mu_{f\restriction_m}={\mu_f}\restriction_m$,
\item[(v$_n$)] if $G$ is metric, then $\mathrm{diam}\ U_f\le 1/2^n$ for every $f\in 2^n$.
\end{itemize}
For $n=0$ define $U_\emptyset=G$. Define also $b_{\emptyset,0}=e$. Then (i$_0$), (ii$_0$),
(iii$_0$), (iv$_0$) and (v$_0$) are trivially satisfied.

Suppose now that 
a family $\{U_f:f\in 2^k\}$ of open subsets of $G$ and
a family $\{\mu_f: f\in 2^k\}$ of order-preserving injections from $k$ to
$\N$
have already been constructed so that the 
properties (i$_k$), (ii$_k$),
(iii$_k$), (iv$_k$) and (v$_k$)
hold
for every integer $k$ with $0\le k\le n$.
Let us define a family $\{U_f:f\in 2^{n+1}\}$ of open subsets of $G$ and a family
$\{\mu_f: f\in 2^{n+1}\}$ of order-preserving injections from $n+1$ to
$\N$ satisfying properties (i$_{n+1}$), (ii$_{n+1}$), (iii$_{n+1}$) and (iv$_{n+1}$) .

Fix $f\in 2^n$.
From \eqref{limit:is:e} it follows that $\lim_{k\to\infty} b_{f,n}a_k=b_{f,n}$.
Since $U_f$ is an open subset of $G$ and $b_{f,n}\in U_f$ by (ii$_n$),
we can find $j\in\N$ such that $b_{f,n}a_k\in U_f$ whenever $k\ge j$. Choose distinct $m_0,m_1\in \N$ such that $m_l\ge j$ and
\begin{equation}
\label{choosing:the:biggest}
m_l>\max\{\mu_f(i):i\in n\} \mbox{ for } l=0,1.
\end{equation}
Define $\mu_{f\hat{\mbox{ }}l}=\mu_f\cup\{\langle n, m_l\rangle\}$ for $l=0,1$,
where $f\hat{\mbox{ }}l=f\cup\{\langle n, l\rangle\}\in2^{n+1}$.
Since \eqref{choosing:the:biggest} holds and $\mu_f$ is an order-preserving injection
from $n$ to $\N$, it follows that each $\mu_{f\hat{\mbox{ }}l}$ is an order-preserving
injection from $n+1$ to $\N$. Clearly, $\mu_{f\hat{\mbox{}}l}\restriction_n=\mu_f$,
which together with (iv$_k$) for all $k\le n$ gives (iv$_{n+1}$).

By our construction,
\begin{equation}
\label{bs:in:Uf}
b_{f,n}a_{m_0}\in U_f
\mbox{ and }
b_{f,n}a_{m_1}\in U_f.
\end{equation}
Since $m_0\not=m_1$, we have $a_{m_0}\not=a_{m_1}$, and thus
$b_{f,n}a_{m_0}\not=b_{f,n}a_{m_1}$. Therefore, we can choose open sets $U_{f\hat{\mbox{ }}0}$
and $U_{f\hat{\mbox{ }}1}$ such that 
(v$_{n+1}$) holds,
$U_{f\hat{\mbox{ }}0}\cap U_{f\hat{\mbox{ }}1}=\emptyset$ and
\begin{equation}
\label{b_{f,n}}
b_{f,n}a_{m_l}\in U_{f\hat{\mbox{ }}l}\subseteq \overline{U_{f\hat{\mbox{ }}l}}
\subseteq U_f
\mbox{ for }
l=0,1.
\end{equation}
In particular, (i$_{n+1}$) holds. Since $a_{m_l}=a_{\mu_{f\hat{\mbox{ }}l}(n)}$,
$$
b_{f,n}a_{m_l}=b_{f,n}a_{\mu_{f\hat{\mbox{ }}l}(n)}=\left(\prod_{i=0}^{n-1} a_{\mu_f(i)}\right) a_{\mu_{f\hat{\mbox{ }}l}(n)}
=
\prod_{i=0}^{n} a_{\mu_{f\hat{\mbox{ }}l}(i)}
=
b_{f\hat{\mbox{ }}l,n+1}
$$
for $l=0,1$. Combining this with \eqref{b_{f,n}}, we obtain (ii$_{n+1}$) and (iii$_{n+1}$). The inductive construction is completed.
\begin{claim}
\label{claim1}
For every $f\in 2^\N$ there exists 
$x_f\in F_f=\bigcap\{\overline{U_{f\restriction_n}}:n\in\N\}\not=\emptyset$.
\end{claim}

\begin{proof}
Define $\mu_f=\bigcup\{\mu_{f\restriction_k}:k\in\N\}$. Since (iv$_k$) holds for
every $k$, and each $\mu_{f\restriction_k}$ is an order-preserving injection from $k+1$ to $\N$,
we conclude that $\mu_f$ is an order-preserving injection from $\N$ to $\N$. Since $A$ is an
$f_1^\star$-productive sequence, there exists a limit
\begin{equation}
\label{defining:x_f}
x_f=\lim_{k\to\infty} \prod_{i=0}^{k-1} a_{\mu_{f(i)}}.
\end{equation}

Fix $n\in\N$. For every $k\in\N$ with $k\ge n$, we have
$$
\prod_{i=0}^{k-1} a_{\mu_{f(i)}}
=
\prod_{i=0}^{k-1} a_{\mu_{f\restriction_k(i)}}
=
b_{f\restriction_k,k}\in U_{f\restriction_k}\subseteq U_{f\restriction_n}
$$
by (ii$_k$) and (iii$_k$). Combining this with \eqref{defining:x_f}, we obtain that $x_f\in \overline{U_{f\restriction_n}}$. Since
$n\in\N$ was chosen arbitrarily, we conclude that $x_f\in F_f\not=\emptyset$.
\end{proof}

Since (i$_n$) and (iii$_n$) hold for all $n\in\N$, one has 
$F_f\cap F_g=\emptyset$ (and so
$x_f\not=x_g$ as well)
whenever $f,g\in 2^\N$ and $f\not=g$.
Combining this with Claim \ref{claim1}, we conclude that $|G|\ge |\{x_f:f\in2^\N\}|=|2^\N|=\cont$.

Assume now that $G$ is metric. 
Since (v$_n$) holds for every $n\in\N$, 
from Claim \ref{claim1}, we conclude that 
each set $F_f$ must become a singleton $\{x_f\}$. This allows us to define a continuous injection $\psi:2^{\N}\to G$ from the
Cantor set $2^\N$ to $G$ given by
$\psi(f)=x_f$ for every $f\in 2^\N$.
\end{proof}

\begin{corollary}\label{|G|<cont:is:NAP}
A topological group of size $<\cont$ does not contain an $f$-productive sequence for any function $f:\N\to(\omega+1)\setminus\{0\}$.
\end{corollary}

\begin{corollary}\label{|G|<cont:is:TAP}
Every topological group of size $<\cont$ is TAP.
\end{corollary}

The following example shows that ``$f$-productive'' can not be
replaced by ``$f$-Cauchy productive'' in Corollary
\ref{|G|<cont:is:NAP}.

\begin{example}\label{NACP:size<c}
Consider the group $\Z$ equipped with the p-adic topology $\tau_p$.
Then $G=(\Z,\tau_p)$ is a countable metric group.
Since $G$ is
linear and non-discrete, it is not NSS. Consequently, $G$ contains an
$f_\omega$-Cauchy productive sequence (even an $f_\omega$-productive set) by Theorem
\ref{metric:NACP:iff:NSS}.
\end{example}

\begin{remark}
Combining Corollary \ref{no:f-product:seq:in:nondiscrete:metr:group} and 
Theorem \ref{size:of:f_1-prductive:groups>c}, 
we obtain
the well-known fact that every non-discrete  Weil complete metric group 
contains a homeomorphic copy of the Cantor set.
\end{remark}

\section{Existence of $f$-productive sequences for various functions $f:\N\to\omega+1$}

In this section we consider the question of existence of $f$-productive sequences for various functions $f:\N\to\omega+1$. The first step in addressing this problem was made already in Lemma \ref{if:f<g}.
For abelian groups
even a stronger version of Lemma \ref{if:f<g} holds.

\begin{proposition} Let $f:\N\to \omega\setminus\{0\}$ and $g:\N\to\omega\setminus\{0\}$ be functions such that $g\le^* kf$ for some positive
integer $k$. Then every $f$-summable sequence
in an abelian topological group
$G$
is $g$-summable in $G$ as well.
\end{proposition}

\begin{proof}Assume that  $\{a_n:n\in\N\}$ is an $f$-summable sequence in $G$. By Lemma \ref{if:f<g} it suffices to prove that it is also
$kf$-summable. Let $\{z_n:n\in\N\}$ be a sequence of natural numbers such that $z_n\leq kf(n)$ for every $n\in\N$.  For
$i\in\{1,\ldots,k\}$ and $n\in\N$ define $$z_{in}=\left\{\begin{array}{ll}
                                       f(n) & \mbox{if $if(n)\leq z_n$}\\
                                       z_n-(i-1)f(n) & \mbox{if $(i-1)f(n)\leq
                                       z(n)<if(n)$}\\
                                       0 & \mbox{if $z_n<(i-1)f(n)$}.
                                    \end{array}
                           \right. $$
Clearly, $z_{in}\leq f(n)$ and $z_n=\sum_{i=1}^kz_{in}$. Since $\{a_n:n\in\N\}$ is $f$-summable,
the sequence $\{z_{in}a_n:n\in\N\}$ is summable for every $i=1,\ldots,k$.
Using Lemma \ref{sum:of:lin:is:lim:of:sum} $k$ many times,
we conclude that the sequence $\{z_n a_n:n\in\N\}$
is summable as well.
This shows that $\{a_n:n\in\N\}$ is
$kf^\star$-summable sequence in $G$. Hence, it is $kf$-summable by Proposition \ref{fstar:is:f:for:abel}.
\end{proof}

\begin{corollary}
\label{bounded:f}
If $f:\N\to\N\setminus\{0\}$ is a bounded function,
then
a
sequence of elements of an abelian topological group
is $f$-summable if and only if it is $f_1$-summable.
\end{corollary}

Returning back to the general case, we prove 
the following
\begin{proposition}
If a topological group $G$ contains an $f$-productive sequence (an $f$-Cauchy productive sequence) for some unbounded function $f:\N\to \omega$, then $G$ contains a $g$-productive sequence (a $g$-Cauchy productive sequence) for every 
function $g:\N\to\omega+1$ such that the set $\{n\in\N: g(n)=\omega\}$ is finite.
\end{proposition}
\begin{proof}
By our assumption on $g$, there exists $m\in\N$ such that $g(n)\in\omega$ for all $n\in\N$ with $n\ge m$.
Let $\{a_k:k\in\N\}$ be an $f$-productive sequence (an $f$-Cauchy productive sequence) in $G$ for an unbounded function $f:\N\to \omega$. 
Since $f$ is unbounded, by induction on $n\in\N$ we can choose 
a strictly increasing sequence $\{k_n:n\in\N,n\ge m\}\subseteq\N$ such that
$f(k_n)\ge g(n)$ for every $n\in\N$ with $n\ge m$. Define $h:\N\to\N$ by $h(n)=f(k_n)$ for all $n\in\N$. Since $\{a_k:k\in\N\}$ is $f$-productive ($f$-Cauchy productive), the subsequence $\{a_{k_n}:n\in\N, n\ge m\}$ of the sequence $\{a_k:k\in\N\}$ is $h$-productive ($h$-Cauchy productive). Since $g\le^*h$, from Lemma 
\ref{if:f<g}
we conclude that the sequence $\{a_{k_n}:n\in\N\}$ in $G$ is $g$-productive ($g$-Cauchy productive as well).
\end{proof}

\begin{proposition}
Let $g:\N\to\omega+1$ be a function such that the set $\{n\in\N: g(n)=\omega\}$
is infinite. If a topological group $G$ contains a $g$-productive ($g$-Cauchy productive) sequence, then $G$ contains also an $f_\omega$-productive (an $f_\omega$-Cauchy productive) sequence.
\end{proposition}
\begin{proof}
Let $\{a_k:k\in\N\}$ be an $g$-productive sequence (an $g$-Cauchy productive sequence) in $G$. 
Choose a strictly increasing sequence $\{k_n:n\in\N,n\ge m\}\subseteq\{n\in\N: g(n)=\omega\}$.
Since $\{a_k:k\in\N\}$ is $g$-productive ($g$-Cauchy productive), the subsequence $\{a_{k_n}:n\in\N, n\ge m\}$ of the sequence $\{a_k:k\in\N\}$ is $f_\omega$-productive ($f_\omega$-Cauchy productive). 
\end{proof}

Any metric abelian group $G$ of size $<\cont$ does not contain
$f_1$-summable sequences by Corollary \ref{|G|<cont:is:NAP}.
If an abelian group $G$ contains an $f_1$-summable sequence, then it has an
$f$-summable sequence for every bounded function $f:\N\to\omega$ (Corollary \ref{bounded:f}).
\begin{example}
\label{bounded:functions}
Let $G=\{g\in\Z^N: \exists\ k\in\N\ |g|\le k\}$ be the subgroup of the linear metric group $\Z^\N$ (equipped with the Tychonoff product topology).
Then {\em $G$ contains an $f$-summable sequence for every bounded function $f:\N\to\omega$, yet it does not contain  $g$-summable sequences for any unbounded function $g:\N\to\omega$\/}. 
\end{example}
The group of real numbers $\R$ contains $g$-summable sequences for all functions
$g:\N\to\omega$ by Corollary \ref{no:f-product:seq:in:nondiscrete:metr:group},
yet it does not contain any $f_\omega$-summable sequence by 
Corollary \ref{weil:NSS:iff:TAP}. 

This classification demonstrates that there are precisely 
four distinct classes of abelian topological groups related to the existence 
of $f$-summable sequences, and examples witnessing that they are distinct can all be chosen to be metric groups.

\begin{remark}
Example \ref{bounded:functions} shows that:
\begin{itemize}
\item[(i)]
``$f$-Cauchy productive'' cannot be replaced by ``$f$-productive''
in Theorems \ref{non1-TAP} and \ref{AP-Cauchy:seq:in:lin:groups},
\item[(ii)]
Weil completeness cannot be omitted 
in
Corollaries \ref{no:f-product:seq:in:nondiscrete:metr:group},
 \ref{Weil:complete:linear:group} and
\ref{linear:complete:corollary}, and
\item[(iii)]
sequential Weil completeness is essential in Corollaries
\ref{NACP:iff:no:conv:sequences} and 
\ref{sequentially:complete:liner:TAP:group:is:discrete}.
\end{itemize}
\end{remark}

\section{Basic properties of the class of TAP groups}

\begin{fact}[\cite{Sp}]\label{cont:hom:image:of:AP:is:AP} The image
of an $f_\omega$-productive set under a continuous homomorphism is an $f_\omega$-productive set.
\end{fact}

\begin{lemma}\label{trivial:lemma}
Let $G$ be a TAP group and $A=\{a_n\in
G\setminus\{e\}:n\in\mathbb{N}\}$ a sequence of (not necessarily
distinct) elements of $G\setminus\{e\}$. Then there exists a
bijection $\varphi:\N\to\N$ and a sequence $\{z_n:n\in\mathbb{N}\}$
of integers such that the sequence $\{\prod_{n=0}^k
a_{\varphi(n)}^{z_n}:k\in\N\}$ does not converge.
\end{lemma}

\begin{lemma}\label{char:of:ap:in:K} Let $G=\prod_{t\in T}G_t$, where $G_t$ is
a TAP group for every $t\in T$. Then a faithfully indexed family
$\{a_n:n\in\N\}\subset G$ is $f_\omega$-productive in $G$ if and
only if the set $\{n\in\N:a_n(t)\neq e\}$ is finite for every $t\in
T$.
\end{lemma}

\begin{proof} To prove the ``if'' part, take a
bijection $\varphi:\N\to\N$, a sequence
$\{z_i:i\in\mathbb{N}\}\subset\mathbb{Z}$, and assume that
$\{n\in\N: a_n(t)\neq e\}$ is finite for every $t\in T$. It follows
that the sequence $\{\prod_{i=0}^k a_{\varphi(i)}^{z_i}(t):k\in\N\}$
converges to some $g(t)$ for every $t\in T$. Define $g=(g(t))\in G$,
and observe that the sequence $\{\prod_{i=0}^k
a_{\varphi(i)}^{z_i}:k\in\N\}$ converges to $g$.

To prove the converse, assume that there exists $t\in T$ such that
$\{n\in\N: a_n(t)\neq e\}$ is infinite. Since $G_t$ is TAP, by Lemma
\ref{trivial:lemma} there exists a bijection $\varphi:\N\to\N$ and a
sequence $\{z_n:n\in\mathbb{N}\}$ of integers such that the sequence
$\{\prod_{n=0}^ka_{\varphi(n)}^{z_n}(t):k\in\N\}$ does not converge.
Consequently, the sequence
$\{\prod_{n=0}^ka_{\varphi(n)}^{z_n}:k\in\N\}$ does not converge and
so $\{a_n:n\in\N\}$ is not \ap.
\end{proof}

Our next theorem collects some useful properties of the class TAP:

\begin{theorem}\label{direct:sum:of:TAP:is:TAP} 
\begin{itemize}
\item[(i)]
if $H$ is a closed normal subgroup of a topological group $G$ such that 
both $H$ and $G/H$ are TAP, then $G$ is TAP as well. 
\item[(ii)] 
a subgroup of a TAP group is TAP.
\item[(iii)] 
if $\{G_i:i\in I\}$ is a family of TAP groups, then the direct sum 
$\bigoplus_{i\in I} G_i$ is also TAP. In particular, 
the class of TAP groups is closed under taking finite products.
\end{itemize}
\end{theorem}

\begin{proof} (i) Let $G$ be a topological group and $H$ its closed normal
TAP subgroup such that the quotient group $G/H$ is TAP. Take an infinite set $A\subset G$. There are two cases.

{\sl Case 1\/}. {\it There exists $g\in G$ such that $A\cap gH$ is infinite.}
If $g=e$, then $A$ cannot be $f_\omega$-productive, since $H$ is TAP and closed
in $G$, and thus no infinite $f_\omega$-productive subset of $G$ can sit in $H$.
If $g\neq e$, then $A$ cannot be $f_\omega$-productive either, because every faithfully
indexed sequence of elements of $A$ would have to converge to $e$ by
Lemma \ref{AP:is:null:sequence}. Since $e\not\in gH$ and $gH$ is closed (because $H$ is), this is not possible.

{\sl Case 2\/}. {\it $A\cap gH$ is finite for every $g\in G$.}
In this case we can fix a set $A'=\{a_n:m\in\N\}\subset A$ such that $a_mH=a_nH$
if and only if $m=n$. Clearly then $p(a_m)=p(a_n)$ if and only if $m=n$,
where $p:G\to G/H$ is the quotient mapping.
Thus $p(A')$ is infinite. 
Since $G/H$ is TAP, $p(A')$ cannot be  $f_\omega$-productive.
By Fact \ref{cont:hom:image:of:AP:is:AP}, $A'$ is not
$f_\omega$-productive. Since $A'\subseteq A$, the bigger set cannot be $f_\omega$-productive.

In both cases $A$ is not $f_\omega$-productive set in $G$. Since $A$ was chosen arbitrarily, we conclude that $G$ is TAP.
\medskip

 (ii) is obvious.
\medskip

(iii) Let $G=\bigoplus_{i\in I} G_i$, where $I$ is a set and $G_i$ is a
nontrivial TAP group for every $i\in I$. Pick a faithfully indexed
set $A=\{a_n:n\in\N\}\subset G$. We are going to show that $A$ is not
\ap. There are two cases.

{\sl Case 1\/}. {\it There exists $i\in I$ such that the set
$\{n\in\N:a_n(i)\neq e\}$ is infinite}. Since $G_i$ is TAP, by Lemma
\ref{trivial:lemma} there exists a bijection $\varphi:\N\to\N$ and a
sequence $\{z_n:n\in\mathbb{N}\}$ of integers such that the sequence
$\{\prod_{n=0}^k a_{\varphi(n)}^{z_n}(i):k\in\N\}$ does not
converge. Consequently, the sequence
$\{\prod_{n=0}^ka_{\varphi(n)}^{z_n}:k\in\N\}$ does not converge,
and thus, the set $A$ is not \ap.

{\sl Case 2\/}. {\it For every $i\in I$ the set
$\{n\in\N:a_n(i)\neq e\}$ is finite.} Since for every $g\in G$ the set
$s(g)=\{i\in I:g(i)\neq e_i\}$ is finite, one can easily find a
faithfully indexed set $B=\{b_n:n\in\N\}\subset A$ such that the
family $\{s(b_n):n\in\N\}$ is pairwise disjoint. Assume that $A$ is
\ap. Then $B$ is \ap\ as well, and thus the sequence
$\{\prod_{n=0}^kb_n:k\in\N\}$ converges to some $g\in G$. On the
other hand, $s(g)=\bigcup_{n\in\N}s(b_n)$ is infinite and thus
$g\not\in G$. This is a contradiction.
\end{proof}

\begin{corollary}\label{opens:subgroups}
If $N$ is an open normal subgroup of a topological group $G$, then
$G$ is TAP if and only if $N$ is TAP. \end{corollary}

\begin{proof}
If $G$ is TAP, then $N$ is TAP by item (ii) of Theorem
\ref{direct:sum:of:TAP:is:TAP}. If $N$ is TAP, the item (i)  of
Theorem \ref{direct:sum:of:TAP:is:TAP} implies that $G$ is TAP, as
the discrete quotient group $G/N$ is TAP.
\end{proof}

The next fact was proved in \cite{Sp}. Its first part also follows
immediately from Theorem \ref{NSS:is:NACP}, and its second part from
Corollary~\ref{no:sequences:no:Cauchy:prod:sequences}.
\begin{fact} [\cite{Sp}]
\label{NSSjeTAP}\label{no:conv:seq:is:TAP}
(a) An NSS group is TAP.

(b) Every topological group without nontrivial convergent sequences is TAP.
\end{fact}

Recall that a topological group is called {\em NSnS} provided that there exists a neighborhood $U$ of the neutral element $e$ that contains no closed normal subgroup beyond $\{e\}$.

The following example demonstrates that, compared to NSS property, the weaker NSnS property does not imply the TAP property.

\begin{example}\label{S(X):nonTAP}
For every infinite set  $X$,  the permutation group  $G = S(X)$,
equipped with the pointwise convergence topology,
is not a TAP group,
because it contains a subgroup
isomorphic to an infinite product of nontrivial topological groups.
Note that even if $G$ fails to be NSS, it is NSnS (as it has no
proper closed normal subgroups at all).
\end{example}

The converse implication in Fact \ref{NSSjeTAP}(a) does not hold in general \cite{Sp}. The main goal of the rest of this paper is to
investigate classes of groups in which the converse implication in (a)  does hold.
This usually happens when some sort of compactness condition is imposed on the group.

\section{TAP property in $\sigma$-compact abelian groups}

\begin{lemma}\label{finitely:generated:are:discrete} Let $F(X)$ be a free (abelian) topological group over a Tychonoff
space $X$. Then every finitely generated subgroup of $F(X)$ is discrete and closed in $F(X)$.
\end{lemma}
\begin{proof}
Fix faithfully indexed $\{x_1,\ldots,x_n\}\subset X$. Since $X$ is Tychonoff, for every $i\in\{1,\ldots,n\}$ there exists continuous
function $f_i:X\to\R$ such that $f_i(x_i)=1$ and $f_i(x_j)=0$ for $j\neq i$. Let $\hat{f_i}$ denote the unique continuous homomorphism
extending $f_i$. Then $U=\bigcap_{i=1}^n\hat{f}_i^{-1}((-\frac{1}{2},\frac{1}{2}))$ is an open set in $F(X)$ such that $U\cap\langle\{x_1,\ldots,x_n\}\rangle=0$. Thus the subgroup
$\langle\{x_1,\ldots,x_n\}\rangle$ of $F(X)$ is discrete. To show that it is closed, take $a\in F(X)\setminus\langle\{x_1,\ldots,x_n\}\rangle$. Then $a=\prod_{i=0}^ky_i^{z_i}$, where $z_i$ is a nonzero integer and $y_i\in X$ for every $i=0,\ldots,k$, and $y_m\not\in\{x_1,\ldots,x_n\}$ for some $m\in\{0,\ldots,k\}$. Since $X$ is Tychonoff, there exists a continuous function $f:X\to\R$ such that $f(x)=0$ for every $x\in\{x_1,\ldots,x_n\}\cup\{y_0,\ldots,y_k\}\setminus\{y_m\}$, and $f(y_m)=1$. Let $\hat{f}:F(X)\to\R$ be the unique continuous
homomorphism extending $f$. Obviously, $f^{-1}(\R\setminus\{0\})$ is an open set separating $a$ from $\langle\{x_1,\ldots,x_n\}\rangle$.
Since every finitely generated subgroup of $F(X)$ is a subgroup of $\langle K\rangle$ for some finite $K\subset X$, the proof is complete.
\end{proof}

\begin{proposition}\label{prop:free:guy:is:TAP}
Let $F(X)$ be the free (abelian) topological group over a  space $X$ (in the sense of Markov). If $X$ is compact, then $F(X)$ is TAP, Raikov complete and $\sigma$-compact.
\end{proposition}
\begin{proof}
It is clear that $F(X)$ is $\sigma$-compact.  The fact that $F(X)$ is Raikov complete can be found in \cite[Corollary 7.4.12]{Archang}.
To show that $F(X)$ is \IANS, take an infinite set $A\subset F(X)$. There are two cases.

{\sl Case 1\/}. {\it There exist finite set $K\subset X$ such that $A\subset \langle K\rangle$ \/}. By Lemma
\ref{finitely:generated:are:discrete}, $\langle K\rangle$ is discrete and closed in $F(X)$. Thus $A$ is not \ap\ set in $\langle
K\rangle$. Since the latter group is closed in $F(X)$, $A$ is not an \ap\ set in $F(X)$ as well.

{\sl Case 2\/}. {\it For every finite set $K\subset X$ we have $A\not\subset \langle K\rangle$\/}. Every compact subset of $F(X)$, in particular, every convergent sequence, is contained in some $B_n$, where $B_n=\{\sum_{i=1}^nz_ix_i:z_1,\ldots,z_n\in\mathbb{Z},x_1,\ldots,x_n\in X\}$; see \cite[Corollary 7.4.4]{Archang}. Let $A=\{a_i:i\in\mathbb{N}\}$ be a (faithful) enumeration of $A$. Put $s_n=\sum_{i=1}^na_i$. Then $\{s_n:n\in\mathbb{N}\}\not\subset B_m$
for any $m\in\mathbb{N}$. Thus the sequence $\{s_n:n\in\mathbb{N}\}$ can not converge. Consequently, $A$ is not \ap.
\end{proof}

\begin{theorem}
Let $X$ be a compact non-metrizable space and $A(X)$ a free abelian topological group over $X$. Then $A(X)$ is complete $\sigma$-compact \IANS\ group that is not NSS.
\end{theorem}

\begin{proof} Since every NSS abelian group has a weaker metric topology and $X$ is a subspace of $A(X)$, it follows that $A(X)$ is not NSS. It
remains to use Proposition \ref{prop:free:guy:is:TAP}.
\end{proof}

\section{TAP property in (locally) compact, $\omega$-bounded, initially $\omega_1$-compact and countably compact groups}

\begin{theorem}
\label{totally:diconnected:locally:compact:TAP:is:discrete} For a totally disconnected locally compact group $G$  the following are equivalent:
\begin{itemize}
  \item[(a)]  G  is discrete;
  \item[(b)] G is NSS;
  \item[(c)] G is TAP;
  \item[(d)] G has no nontrivial convergent sequences.
\end{itemize}
\end{theorem}
\begin{proof} The implication (a)$\to$(b) is obvious, the implication (b)$\to$(c) is Fact \ref{NSSjeTAP}. Since  totally disconnected locally compact
groups are complete linearly topologized, the implication (c)$\to$(d) follows from Theorem \ref{char:of:complete:linear:TAP:groups}. Finally, the implication
(d)$\to$(a) follows from the well known fact that a totally disconnected locally compact group must have an open compact subgroup. Let $K$ denote this open compact subgroup of $G$. Then $G$ is discrete if and only if $K$ is discrete (i.e., infinite). Since infinite compact groups have nontrivial convergent sequences, we conclude that  (d)$\to$(a). This proves the corollary.
\end{proof}

In the sequel we apply several times the  following theorem of Davis \cite{Dav}: 

\begin{theorem}\label{theorem:Davis}
Every locally compact group $G$ is {\em homeomorphic} to a product $K \times \R^n \times D$, where $K$ is a compact subgroup of $G$, $n\in \N$ and $D$ is a discrete space.
\end{theorem}

\begin{theorem}\label{zerodim:TAP:compact:is:finite}
A a non-discrete locally compact group $G$ contains an $f$-productive set for every  for every function $f:\N\to \omega\setminus\{0\}$. Moreover, if $G$ is 
zero-dimensional, then $G$ contains an $f_\omega$-productive set (i.e., $G$ is not TAP).
\end{theorem}

\begin{proof}
If $G$ is zero-dimensional, then Theorem \ref{totally:diconnected:locally:compact:TAP:is:discrete} implies that  $G$ is not TAP, i.e., 
$G$ contains an $f_\omega$-productive set.

In the general case we consider two cases. If $G$ is metrizable, then the assertion follows from Theorem \ref{non1-TAP}, since locally compact groups are Weil complete. 
Assume $G$ is not metrizable. By Theorem \ref{theorem:Davis} $G$ is homeomorphic to a product $K \times \R^n \times D$, where $K$ is a compact subgroup of $G$, $n\in \N$ and $D$ is a discrete space. Our assumption on $G$ implies that the compact subgroup $K$ is not metrizable. If $K$ has a non-torsion element $a$, then the closed subgroup $N$ generated by $a$ is a compact abelian group, so by Rudin's theorem $N$ contains an infinite compact metrizable subgroup $M$. To $M$ we can apply again Theorem \ref{non1-TAP} to find
an $f_\omega$-productive set in $M$. If $K$ is a torsion group, then $K$ is zero-dimensional, so $K$ contains an $f_\omega$-productive set by the first part of the argument. 
\end{proof}

This theorem implies that a locally compact group is discrete iff it does not contain $f_1$-productive sequence. Moreover, 
 every infinite compact group contains an $f$-productive sequence for every  for every function $f:\N\to \omega\setminus\{0\}$.

Our next example shows that compactness cannot be weakened to precompactness in Corollary \ref{zerodim:TAP:compact:is:finite}. It also shows that sequential completeness cannot be omitted in Corollary \ref{sequentially:complete:liner:TAP:group:is:discrete}, and completeness cannot be dropped in Corollary \ref{linear:complete:corollary}.

\begin{example}\label{question1}
\rm{ Let $p$ be a prime number and let $\tau_p$ denote the $p$-adic topology of $\Z$. Then {\em the group $G=(\Z,\tau_p)$ is an infinite precompact (thus non-discrete) linear (thus, not NSS) metric TAP group\/}. Indeed, $G$ is TAP by Corollary \ref{|G|<cont:is:TAP}. }
\end{example}

Let us recall that the Lie groups are precisely the locally Euclidean topological groups. Another their equivalent description involving the NSS property is this: 
a locally compact group $G$ is a Lie group if and only if it is NSS. 

\begin{corollary}\label{locally:compact:abelian:TAPgroup} Let $G$ be a locally compact abelian TAP group.
  Then $G \cong \R^m\times \T^n \times D$ for some $m, n\in \N$ and some discrete abelian group $D$. In particular, $G$ is a Lie group.
\end{corollary}

\begin{proof}
It is known that $G$ has the form $G \cong \R^m\times G_0$, where the group $G_0$ contains an open compact subgroup $K$ \cite{DPS}.
Since subgroup of a TAP group is TAP, from Corollary \ref{compact:abelian:group:are:Lie} we conclude that $K$ is a Lie group, i.e.,
$ K\cong \T^n \times F$ for some $n\in \N$ and some finite group $F$. In the sequel we identify $\T^n$ with an open subgroup of $K$, so that
$\T^n$ is an open subgroup of $G_0$. Since $\T^n$ is divisible, we can write $G_0=\T^n \times D$ for some discrete abelian subgroup $D$.
Thus $G \cong \R^m\times \T^n \times D$. In particular, $G$ is a Lie group.
\end{proof}

In the sequel we need the following two results proved in \cite{DSS}.

\begin{lemma}
\label{omega-bounded:lemma} {\rm \cite{DSS}}
Let $G$ be a topological group such that the closure of every countable subgroup of $G$ is a compact Lie group. Then $G$ is a compact Lie group.
\end{lemma}

Recall that a topological group $G$ is {\em $\omega$-bounded\/} if the closure of every countable subset of $G$ is compact.

\begin{theorem}
\label{second:characterization} {\rm \cite{DSS}}
For every topological group $G$, the following conditions are equivalent:
\begin{itemize}
\item[(i)]
$G$ is a compact Lie group;
\item[(ii)]
$G$ is an $\omega$-bounded group such that all closed totally disconnected subgroups of $G$ are finite.
\end{itemize}
\end{theorem}

\begin{theorem}
\label{compact:TAP:iff:NSS:iff:Lie}
 For a locally compact group $G$ the following are equivalent:
\begin{itemize}
\item[(i)] $G$ is NSS;
\item[(ii)] $G$ is a Lie group.
\item[(iii)] $G$ is TAP.
\end{itemize}
\end{theorem}

\begin{proof} The implication  (i) $\to $ (iii) is Fact \ref{NSSjeTAP}
and the implication  (ii) $\to $ (i) is a well known fact.

 (iii) $\to $ (ii) According to Theorem \ref{theorem:Davis} there exists a compact  subgroup $K$ of $G$ such 
that $G$ is homeomorphic to a product $K \times \R^n \times D$, where  $n\in \N$ and $D$ is a discrete space. 
Since  Lie groups are precisely the locally Euclidean topological groups,
 it is clear that $G$ is a Lie group whenever the compact group $K$ is a Lie group. As a subgroup of the TAP group $G$, $K$ is also TAP, so by Corollary \ref{zerodim:TAP:compact:is:finite}, every closed totally disconnected subgroup of $K$ must be finite. Applying the implication (ii)$\to$(i) of Theorem \ref{second:characterization}, we conclude that $K$ must be a Lie group. This implies that the whole group $G \approx K \times \R^n \times D$ is a Lie group.
\end{proof}

\begin{theorem}
\label{omega:bounded:abelian:group:TAP:iff:NSS}
 For an $\omega$-bounded group $G$ the following statements are equivalent:
\begin{itemize}
\item[(i)] $G$   is TAP;
\item[(ii)] $G$ is a compact Lie group;
\item[(iii)] $G$  is NSS.
\end{itemize}
\end{theorem}
\begin{proof}
(i)$\to$(ii) Let $D$ be a countable subgroup of $G$. Since $G$ is $\omega$-bounded,
$K=\overline{D}$ is a compact subgroup of $G$. Being a closed subgroup of a TAP group $G$,
$K$ is TAP. By Theorem \ref{compact:TAP:iff:NSS:iff:Lie}, $K$ is a Lie group. Therefore, $G$ satisfies the assumption of Lemma
\ref{omega-bounded:lemma}. Applying this lemma, we conclude that $G$ is a
compact Lie group.

The implication (ii)$\to$(iii) is well known, and the implication
(iii)$\to$(i) is Fact \ref{NSSjeTAP}. 
\end{proof}

\begin{corollary}
\label{tot:disc:Obounded:TAP:iff:finite}
An infinite totally disconnected $\omega$-bounded group contains an $f_\omega$-productive set.
\end{corollary}

\begin{proof}
Indeed, let $G$ be a group satisfying the assumption of our corollary. Assume that $G$ does not have an $f_\omega$-productive set.
That is, $G$ is TAP. Then $G$ is compact by Theorem \ref{omega:bounded:abelian:group:TAP:iff:NSS}, and Corollary \ref{zerodim:TAP:compact:is:finite}
yields that $G$ has an $f_\omega$-productive set, giving a contradiction.
\end{proof}

\begin{corollary}
An infinite $\omega$-bounded group has an $f$-productive set for every  for every function $f:\N\to \omega\setminus\{0\}$. 
\end{corollary}

\begin{proof}
Let $G$ be an $\omega$-bounded group, and let $A$ be a countably infinite subset of $G$. Then $A$ is contained in a compact subgroup $K$ of $G$. Applying Theorem \ref{zerodim:TAP:compact:is:finite} to $K$, we obtain a an $f$-productive set in $K$ for every  for every function $f:\N\to \omega\setminus\{0\}$. 
\end{proof}

Call a group $G$ {\em locally $\omega$-bounded\/} if $G$ has a neighborhood $U$ of its identity element such that the closure (in $G$) of each countable subset of $U$ is compact.

One may wonder if a common generalization of Theorems \ref{compact:TAP:iff:NSS:iff:Lie} and \ref{omega:bounded:abelian:group:TAP:iff:NSS} is possible.

\begin{question} Is every locally $\omega$-bounded TAP group an NSS group?
\end{question}

Let us recall two examples from \cite{Sp}.
\begin{example}
\label{example3}
\begin{itemize}
\item[(i)] An infinite pseudocompact group without nontrivial convergent sequences is TAP but not NSS.
\item[(ii)] Let $G$ be any consistent example of a countably compact abelian group groups without nontrivial convergent sequences. Then $G$ is TAP but not NSS.
Since totally disconnected examples with this property exist, this shows that ``$\omega$-bounded" cannot be weakened to ``countably compact" in Corollary \ref{tot:disc:Obounded:TAP:iff:finite}. 
\end{itemize}
\end{example}

Since examples of groups as in Example \ref{example3}(ii) exist under MA, we conclude that countable compactness alone cannot help to establish NSS under the assumption of TAP.
Nevertheless, Example \ref{example3} leaves open the question whether this remains true in ZFC:

\begin{question}\label{question3}
Does there exist a ZFC example of a countably compact TAP group that is not NSS?
\end{question}

Our next theorem shows that any totally disconnected countably compact TAP group {\em must\/} be without nontrivial convergent sequences.

\begin{theorem}\label{tot:disc:cc:TAP:iff:no:conv:seq} For a totally disconnected countably compact group $G$ the following are equivalent:
\begin{itemize}
\item[(a)] $G$ is TAP;
\item[(b)] $G$ has no nontrivial convergent sequences.
\end{itemize}
\end{theorem}
\begin{proof}
Since countably compact groups  are sequentially complete and totally disconnected countably compact groups have linear topology, the conclusion follows from Theorem
\ref{char:of:complete:linear:TAP:groups}.
\end{proof}

We do not even know if strengthening countable compactness to initial $\omega_1$-compactness would help to get NSS property from TAP:
\begin{question}\label{question4}
Does there exist an example of an initially $\omega_1$-compact TAP group  that is not NSS?
\end{question}

\section{A characterization of TAP subgroups of  $\Z_p$}

In this section we substantially strengthen Example \ref{question1} by characterizing subgroups of $\Z_p$ that are TAP; see Corollary \ref{characterization:of:subgroups:of:Z_p:that:are:TAP}.

\begin{proposition}
\label{non-TAP:subgroups:of:Z_p}
Let $p$ be a prime. If $H$ is an infinite non-TAP subgroup of $\Z_p$, then $H$ is an open subgroup of $\Z_p$, so $H=p^n\Z_p$ for some $p$.
\end{proposition}

\begin{proof} First, let us recall, that $\Z_p$ is also a unitary ring, with unit 1 and having as only non-zero ideals the subgroups $p^n\Z_p$ of $\Z_p$ ($n\in \N$). Every $\xi \in \Z_p \setminus p\Z_p$ is an invertible element of the ring $\Z_p$. Any non-zero $\eta\in \Z_p$ can be written as $\eta = p^n \xi$ with $\xi \in  \Z_p \setminus p\Z_p$ and uniquely determined $n\in \N$, that is usually denoted by $v_p(\eta)$.

\begin{claim}
\label{claim:2}
If $\eta, \alpha \in \Z_p$ and $k\in \N$ are such that $v_p(\eta) \geq v_p(\alpha)$ and $k > v_p(\alpha)$, then there exists $z\in \Z$ such that $v_p(\eta - z \alpha) \geq k$.
\end{claim}

\begin{proof}
Since $\Z=\langle 1 \rangle$ is dense in $\Z_p$ and $U= p^k\Z_p$ is open, there exist $c,a \in \Z$ such that $\eta - c\in U$ and $\alpha - a \in U$. Let $c= p^sc_1$ and $a= p^t a_1$, with $c_1, a_1\in \Z\setminus p\Z$. Then $t= v_p(a) = v_p(\alpha)< k$ and $ s= v_p(c) = v_p(\eta) \geq t$. As $p$ does not divide $a_1$, there exists $z\in \Z$ such that $a_1z - p^{s-t}c_1 \in p^{k-t}\Z$. Multiplying by $p^t$ we get
$$
za - c =  az - p^sc_1 \in  p^{k}\Z.
$$
Now
$$
\eta - z\alpha= (\eta - c) + (c-za) + (za - z\alpha) \in p^k \Z_p,
$$
i.e., $v_p(\eta - z\alpha)\geq k$.
\end{proof}

Let $A = (\alpha_i)$ be an $f_\omega$-summable set of $H$ witnessing that $H$ is not TAP. Since $A$ must be a null sequence, by taking eventually a subsequence we
may assume, without loss of  generality, that the sequence $v_p(\alpha_n)$ strictly increases. Let $m= v_p(\alpha_0)$. We shall
prove that $p^m \Z_p \subseteq H$. Take any $\eta \in p^m \Z_p$. Then $v_p(\eta) \geq m= v_p(\alpha_0)< v_p(\alpha_1)$. By Claim
\ref{claim:2} there exists $z_0 \in \Z$ such that $\eta_1= \eta - z_0 \alpha_1$ satisfies $v_p(\eta_1)\geq v_p(\alpha_1)$. Applying again Claim \ref{claim:2} to $\eta_1$ and $\alpha_1$, we find $z_1\in\Z$ such that $v_p(\eta_1 - z_1)\geq v_p(\alpha_2)$, etc. We build by induction a sequence $\{z_i:i\in \N\}$ of integers such that, with $x_k = \sum_{i=1}^kz_i\alpha_i$,
\begin{equation}
\label{eq:dagger1}
v_p(\eta - x_k)\geq v_p(\alpha_{k+1})
\end{equation}
holds for every $k$. By the definition of $f_\omega$-summable set this sequence is convergent in $H$. Let $h= \lim_k x_k\in H$. Since the sequence $\eta-x_k$ converges to 0 by \eqref{eq:dagger1}, we deduce that $\eta = h \in H$. This proves the inclusion  $p^m \Z_p \subseteq H$. Since  $p^m \Z_p$ is open, this implies also that $H$ is open (hence clopen
and consequently compact). Since the quotient $\Z_p /p^m\Z_p$ is isomorphic to the cyclic group $\Z(p^m)$, this implies that $H= p^n \Z_p$ for some natural number $n\leq m$.
\end{proof}

\begin{corollary}
\label{characterization:of:subgroups:of:Z_p:that:are:TAP}
Let $p$ be a prime. A non-zero subgroup $G$ of $\Z_p$ is TAP if and only if $G\not=p^n \Z_p$ for every $n\in\N$.
\end{corollary}

\begin{corollary} Let $p$ be a prime.  Every non-TAP subgroup of $\Z_p$ is isomorphic to $\Z_p$ (so is compact).
\end{corollary}

The above corollary provides many examples of precompact TAP  subgroups of $\Z_p$ of size $\cont$. Indeed, since $\Z_p$ is $q$-divisible for every prime $q\ne p$,   any free subgroup of $\Z_p$ of size $\cont$ will be non-isomorphic to $\Z_p$, hence a TAP group of size $\cont$. 

As an application of our results, we offer a solution of a question from \cite{Sp}.

\begin{example}
Let $p$ be a prime number and $X=\{0\}\cup \{1/n:n\in \N\}$ a convergent sequence. Furthermore, let $G$ be any TAP subgroup of $\Z_p$ (for example, the cyclic group $\Z$ from Example \ref{question1} will do). Then {\em $G$ is a precompact metric TAP group and $X$ is a compact $G^{\star\star}$-regular space such that $C_p(X,G)$ is not TAP.\/}
Indeed, since
$$
\{0\}\times\prod_{n\in\N} p^nG = \{f\in C_p(X,G):
f(0)=0
\mbox{ and }
f({1}/{n})\in p^n G
\mbox { for every }
n\in\N\}\subseteq C_p(X,G)
$$
and each $p^nG$ is nontrivial, $C_p(X,G)$ is not TAP.
\end{example}

This example significantly strengthens \cite[Theorem 6.8]{Sp} (since in that theorem $X$ was only countably compact and $G$ was not metrizable) and answers negatively \cite[Question 11.1]{Sp}.

\section{Groups containing no infinite products }

Here we consider the property of containing infinite products (of nontrivial groups) which is weaker than TAP, i.e., we have the implications

\medskip

\begin{center}
NSS  $\;\;\buildrel{(1)}\over\Longrightarrow \;\;$ TAP $ \;\;\buildrel{(2)}\over\Longrightarrow\;\;$  the group contains no infinite products
\end{center}
\medskip

To the invertibility of (1) was dedicated the first part of the paper, this section is dedicated to the discussion of when (2) is invertible.

\begin{example}\label{example1}
Let $p$ be a prime number and let $\J_p$ denote the group of $p$-adic integers. Then the group $\J_p$ is not TAP by Corollary \ref{linear:complete:corollary}.

 Note that the group $\J_p$ does not even contain a direct product $A \times B$ of two nontrivial subgroups equipped with the product topology.  Indeed, assume that $\J_p$ contains the direct product $P= A \times B$ of two nontrivial subgroups $A$, $B$, each equipped with induced topology and $P$ equipped with the product topology. Then the completion $\widetilde{A}$ of $A$, being isomorphic to a non-zero closed subgroup of $\J_p$, is isomorphic to $\J_p$ itself. Analogously, $\widetilde{B} \cong \J_p$ and $\widetilde{P}\cong \J_p$. On the other hand, $\widetilde{P}= \widetilde{A} \times \widetilde{B} \cong \J_p^2$. To get a contradiction, it remains to see that $\J_p^2\not \cong \J_p$. Indeed, any isomorphism $\J_p^2\cong \J_p$ produces in obvious way isomorphisms $p\J_p^2\cong p\J_p$ and $\J_p^2/p\J_p^2\cong \J_p/p\J_p$. This  leads to a contradiction as $\J_p^2/p\J_p^2\cong \Z(p)^2$ and $\J_p/p\J_p\cong \Z(p)$.
\end{example}

\begin{example} Some non-NSS groups contain infinite products of nontrivial groups. For example the infinite permutation groups $S(X)$ have this property being totally minimal at the same time. (Take a partition $X = \bigcup_{n=1} X_n$ into infinite subsets. Then the infinite product $\prod_n S(X_n)$ is isomorphic to a subgroup of $S(X)$.)
\end{example}

The next theorem gives a complete characterization of the compact abelian groups that contain no infinite products of nontrivial groups.

\begin{theorem} For a compact abelian group $G$ the following are equivalent:
\label{compact:abelian:groups:without:infinite:products}
\begin{itemize}
\item[(a)] $G $ contains no infinite products of nontrivial groups;
\item[(b)] $d=\dim G$ is finite and for every continuous surjective homomorphism $f: G \to \T^d$ one has $\ker f \cong F \times \prod_{k=1}^n\J_{p_k}$, where $F$ is a finite abelian group and $p_1, \ldots, p_n$ not necessarily distinct primes;
\item[(c)] there exists a continuous surjective homomorphism $f: G \to \T^d$ such that $\ker f \cong F \times \prod_{k=1}^n\J_{p_k}$, where $F$ is a finite abelian group and $p_1, \ldots, p_n$ not necessarily distinct primes.
\end{itemize}
\end{theorem}

\begin{proof} (a) $\to$ (b) Obviously, $G$ does not contain the infinite power $\J_p^\omega$ for all primes $p$. Let us see that this implies $d=\dim G< \infty$.
Since $\dim G$ coincides with the free-rank of  the Pontryagin dual $\widehat{G}$, it suffices to see that $r_0(\widehat{G})< \infty$. Arguing by
contradiction we assume that $r_0(\widehat{G})$ is infinite, hence there exists an injective homomorphism $\bigoplus_\omega \Z \to \widehat{G}$. Taking duals we obtain a
a continuous surjective homomorphism $f: G \to \T^\omega$. Since the subgroup $\Z(p^\infty)^\omega$ of $\T^\omega$ contains a copy of the group $\J_p$, we deduce that $\T^\omega$ contains a subgroup $N\cong \J_p^\omega$. Consider the
inverse image $L= f^{-1}(N)$ of $N$ and the surjective restriction $h = f\restriction_L : L \to N$. Taking the Pontryagin duals we obtain an
injective homomorphism $\xi=\widehat{h}: \widehat{N} \to \widehat{L}$ with $\widehat{N} \cong \bigoplus_\omega\Z(p^\infty)$
divisible. Hence the subgroup $\xi(\widehat{N})$ of $\widehat{L}$ splits. Therefore, taking once more the Pontryagin duals, we conclude that $L \cong L_1 \times \J_p^\omega$, i.e., $L$ (hence, $G$ itself) contains the infinite power $\J_p^\omega$ for all primes $p$. This contradiction with (a) proves that $d=\dim G$ is finite.

Now fix an arbitrary continuous surjective homomorphism $f: G \to \T^d$ and let  $N = \ker f$. Since $\dim G = \dim \T^d = d$, it follows by Yamanashita's theorem that $\dim N = 0$. Then $ N = \prod_p N_p$, where each $N_p$ is a pro-$p$-group (\cite[Example 4.1.3(a)]{DPS}). By our hypothesis (a) only finitely many of the groups $N_p$ are non-zero.

Our next step is to determine the structure of the single groups $N_p$. It is known that there exists a closed subgroup $R_p\cong \J_p^{\kappa_p}$ of $N_p$ such that
$\kappa_p$ is a cardinal, $N_p/R_p$ is a product of finite cyclic $p$-groups (in particular, $ t(N_p/R_p)$ is dense in $N_p/R_p$) and the quotient map $\eta: N_p \to N_p/R_p$ satisfies
$$
\eta(t(N_p) )= t(N_p/R_p)\eqno(*)
$$ 
(see \cite{DAGB}). Since $N_p$
is reduced, it contains no copies of $\Z(p^\infty)$. By our hypothesis (a), $N_p[p] \cong \Z(p)^{r_p(N_p)}$ is finite. Thus $r_p(N_p)$ is finite. Then the $p$-group $t(N_p)$ has the form
$t(N_p)\cong \Z(p^\infty)^m\times F_p$, for some $m\in \N$ and a finite $p$-group $F_p$ (\cite{Fu}). Since $t(N_p)$ is reduced as a subgroup $N_p$, it contains no copies of $\Z(p^\infty)$. This yields that $m=0$ and $t(N_p)=F_p$ is finite. Let $k= |t(N_p)|$. Then $kN_p$ is torsion-free and contains $kR_p\cong \J_p^{\kappa_p}$. Since $N_p$ does not contain infinite products, this implies that $\kappa_p$ is finite. By (*), $N_p/R_p= t(N_p/R_p)$ is finite. Hence, $N_p$ is a finitely generated $\J_p$-module of rank $\kappa_p$. Thus $N_p\cong  \J_p^{\kappa_p} \times B_p$ for some some finite $p$-group $B_p$.  This proves (b).

Obviously (b) $\to$ (c).

To prove (c) $\to$ (a) assume for a contradiction that there exists
a subgroup $A$ of $G$ such that  $A\cong P= \prod_n A_n$, where each
$A_n$ is a  nontrivial subgroup of $G$ and $P$ is equipped with the
product topology. Let $j: P \to A$ be this isomorphism. Then we can
extend $j$ to an isomorphism
 $\widetilde{j}:  \prod_n \widetilde{A_n} =\widetilde{P} \to  \widetilde{A} $ of  the completions. Since $G$ is compact, we can identify, without loss of  generality, the completion $\widetilde{A} $ of $A$ with its closure $\overline A$ in $G$. Analogously, we can identify $\widetilde{A_n}$ with the closure $\overline A_n$ of $A_n$ in $G$. In other words, we can replace the initial subgroup $A$ of $G$, isomorphic to an infinite product of nontrivial subgroup, with a {\em closed} subgroup $A$ of $G$, isomorphic to an infinite product of nontrivial closed subgroups of $G$ (equipped with the product topology).  Let $N = \ker f \cong F \times \prod_{k=1}^n\J_{p_k}$, as in (c). Next we note that the group $F \times \prod_{k=1}^n\J_{p_k}$ satisfies the ascending chain condition on closed subgroups since its Pontryagin dual $F \times \bigoplus_{k=1}^n\Z(p_k^\infty)$ satisfies the descending chain condition on subgroups \cite{Fu}.
Let $\bigcup_{n=1}^\infty I_n$ be a partition of $\N$ into infinite subsets $I_n$. For every $n\in \N$ let $B_n = \prod_{m\in I_n}A_m$. Assume that $N_n= B_n \cap N\ne 0$ for all $n\in \N$. Then the subgroups $L_n= N_1 \times N_2 \times \ldots \times N_n$ of $N$ form an infinite ascending chain, a contradiction. Hence some $N_n = 0$. Then the homomorphism
$f: G \to \T^d$ sends the infinite product $B_n$ monomorphically in $\T^d$. Since $B_n$ is compact, $f\restriction_{B_n}: B_n \to f(B_n)$ is an isomorphism, so
$\T^d$ contains infinite products. This contradicts the fact that $\T^d$ is NSS, since no  infinite product of groups can be NSS.
\end{proof}

In the next remark we discuss the relation of the property ``the group $G$ contains no infinite products" with the lattice theoretic properties of the lattice 
$\mathcal L(G)$ of closed subgroups of $G$. 

\begin{remark}
(a) We say that a topological group $G$ satisfies the {\em descending chain condition on closed subgroups } (briefly, D.C.C.) if every descending chain of closed subgroups of $G$ stabilizes. If $G$ is a compact abelian group, the lattice $\mathcal L(G)$
 is anti-isomorphic to the lattice $\mathcal L(\widehat{G} )$, where
  $ \widehat{G} $ is the discrete dual group of $G$. That is why, $G$ satisfies A.C.C. (resp., D.C.C) if and only if $\widehat{G} $ satisfies D.C.C. (resp., A.C.C.). It is well known, that $\widehat{G} $ is A.C.C. if and only if $\widehat{G} $ is finitely generated, while $A$ is D.C.C. if and only if $\widehat{G}  \cong F \times \bigoplus_{k=1}^n\Z(p_k^\infty)$ \cite{Fu}. So, $G$ is D.C.C. if and only if $G\cong \T^n \times F$ if and only if $G$ is NSS.

(b) One can wish to modify the argument of the proof of (c) $\to$ (a) to show that whenever $G$ is a compact topological group such that for some closed normal subgroup $N$ of $G$ satisfying D.C.C. the quotient group $G/N$ satisfies A.C.C., then the group $G$ does not contain of infinite direct products. This can be done as follows. Assume that an infinite
product $A\cong P= \prod_n A_n$ is contained in $G$ and arguing as in the above proof suppose that $A$ and $A_n$ are closed subgroups of $G$. Define $B_n$, $N_n$ and $L_n$ as above. Note that the direct sum $\bigoplus _n N_n= \bigcup _n L_n$ of the subgroups $N_n$ is contained in $N$ and as a subgroup of $P$ carries the product topology. So its closure in $G$ coincides with the product $\prod_n N_n$ and it must be contained in $N$ as $N$ is a closed subgroup. But then $N$ cannot satisfy $D.C.C.$, a contradiction. Hence $N_n$ is trivial for some $n$ and the argument can be concluded as above. The careful reader will notice that we did not use commutativity of $G$ in this argument. In case $G$ is also abelian, one can carry out the following shorter argument. According to item (a), $N\cong \T^n \times F$ is a Lie group. It is well known that the group $\T^n$ splits in every compact abelian group. Hence,
the subgroup $\T^n$ of $G$ splits, i.e., there exists a closed subgroup $L$ of $G$ (containing $F$), such that $G \cong \T^n \times L$. The quotient group $G/\T^n \cong L$
contains a finite subgroup (isomorphic to) $F$, such that $L/F\cong G/N$ is an A.C.C. group. Since $F$ is finite, we conclude that $L$ itself is A.C.C. So $G$ splits
into a direct product of a torus $\T^n$ and an A.C.C. group $F_1 \times \prod_{k=1}^n\J_{p_k}$, where $F_1$ is a finite group (containing $F$). Hence, $G$ is a direct product
of a Lie group $\T^n \times F_1$ and a group of the form $\prod_{k=1}^n\J_{p_k}$.
\end{remark}

The above theorem implies that the compact abelian group containing no infinite products are metrizable. The next corollary shows a more general property:

\begin{corollary}\label{no:infinite:products>metrizable}
Every locally compact abelian group containing no infinite products is metrizable.
\end{corollary}

\begin{proof}
Note that a LCA group $G$ has the form $G = \R^n \times G_0$, where $G_0$ contains an open compact group $K$.  Since compact abelian group containing no infinite products are metrizable by the first part of the proof, we conclude that $G$ is metrizable as well.
\end{proof}

Next we discuss the connection between TAP and NSS for compact abelian group containing no infinite products.

\begin{corollary}\label{infinite:products}
Let $G$ be a compact abelian group containing no infinite products. Then there exists a finite set of primes $\P_G$ and  $k_p\in \N\setminus \{0\}$ for $p\in \P_G$, such that 
$G$ contains a closed subgroup $ N \cong \prod_{p\in \P_G}^n\J_p^{k_p}$,
 such that $G/N$ is NSS (i.e., $G/N$ is a Lie group). In this setting the following are equivalent:
\begin{itemize}
\item[(a)] $G$ is NSS (i.e., $G$ is a Lie group);
\item[(b)]  $G$ is TAP;
\item[(c)]  $\P_G= \emptyset$
\end{itemize}
\end{corollary}

\begin{proof} 
According to Theorem \ref{compact:abelian:groups:without:infinite:products}(c), $d=\dim G< \infty$ and there exists a continuous surjective homomorphism $f: G \to \T^d$ such that $\ker f \cong F \times \prod_{k=1}^n\J_{p_k}$, where $F$ is a finite abelian group and $p_1, \ldots, p_n$ not necessarily distinct primes. Write the product $N=\prod_{k=1}^n\J_{p_k}$
as $N=\prod_{p\in \P_G}^n\J_p^{k_p}$ for an appropriate finite set $\P_G$ of primes and naturals $k_p\in \N\setminus \{0\}$. Note that $G/(F\times N)\cong \T^d$, hence
$G/N \cong F \times \T^d$ is a Lie group.  This concludes the proof of the first assertion. 

Obviously, (a) $\to $ (b).  (b) $\to $ (c). If $G$ is TAP, then $N$ is necessarily finite since no group $\J_p$ is TAP, while a subgroup of a TAP group is TAP. Since $N$ is finite if and only if $\P = \emptyset$, we are done. (b) $\to $ (a). From (c) we deduce that $N$ is finite. Since $G/N$ is a Lie group, we deduce that $G$ is a Lie group as well. Hence $G$ is NSS.
\end{proof}

\begin{corollary}
\label{compact:abelian:group:are:Lie}
 Let $G$ be a compact abelian group. Then $G$ is TAP if and only if $G$ is NSS (i.e., $G$ is a Lie group).
\end{corollary}

\begin{proof} Since both TAP and NSS imply that $G$ contains no infinite products, Corollary \ref{infinite:products} applies. \end{proof}

In the next remark we discuss uniqueness of the set $\P_G$ and the multiplicities $k_p$, as well as of the subgroups $N$.  

\begin{remark}\label{The:filter}
The closed subgroup $N$ in Corollary \ref{infinite:products} is not uniquely determined. 
(Indeed, take $G = \Z_p \times \T$. Then both $N_1= \Z_p \times \{0\}$ and $N_2= pN_1 + \langle (1, a)\rangle$, where $a\in \T$ has $o(a)=p$, satisfy $G/N_1\cong G/N_2\cong \T$.)
The family $\mathfrak J_G$ of all subgroups $N$ of $G$ with this property has following properties:
\begin{itemize}
\item[(a)]  the set $\P_G$ and the multiplicities $k_p$ are uniquely determined by $G$ (i.e., all $N \in \mathfrak J_G$ have the same $\P_G$ and multiplicities $k_p$);
\item[(b)]  if $N_1, N_2 \in{\mathfrak{J}}_G$, then  $N_1\cong N_2$, $N_1\cap N_2 \in \mathfrak{J}_G$ 
 (i.e., $\mathfrak{J}_G$ ``behaves" as a filter), and
$N_1 \cap N_2$ is a finite-index subgroup of $N_1$ and $ N_2$.
\end{itemize}
Inspired by the above corollary and item (ii) above, we shall call any member $N$ of the family $\mathfrak J_G$ {\em TAP-defect} of $G$ (as $G/N$ is TAP, but $G/L$ is not a TAP
group for any closed subgroup $L$ of $N$ of infinite index).
\end{remark}

The next somewhat surprising result shows that a relatively innocent looking property as TAP may have a very strong impact on the algebraic structure of a minimal group.

\begin{proposition}\label{almost:torsion:free}
Let $G$ be a minimal abelian group.  If $G$  contain no infinite
products of nontrivial groups (in particular, if $G$ is TAP), then
$G$ is almost torsion-free.
\end{proposition}

\begin{proof}
We have to show that $r_p(G)$ is finite for all primes $p$. According to Theorem  \ref{minimality-criterion}, $G$ contains the closed subgroup $K[p]$ of its compact completion. As $K[p]$ is topologically isomorphic to a product of copies of $\Z(p)$, our hypothesis on $G$ yields that $K[p]$ is finite. This proves that $r_p(K)=r_p(K)$ are finite.
\end{proof}

\begin{corollary} For minimal torsion group $G$ the following properties are equivalent:
\begin{itemize}
\item[(a)]  $G$  contains no infinite products of nontrivial groups;
\item[(b)]  $G$ is TAP.
\end{itemize}
\end{corollary}

\begin{proof} The implication (b) $\to $ (a) is known.

Assume that (a) holds. Then by Proposition \ref{almost:torsion:free} the group $G$ is almost torsion-free. Since $G$ is also torsion, this yields that $G$ is countable. Hence $G$ is TAP by Corollary  \ref{|G|<cont:is:TAP}. 
\end{proof}

\begin{theorem}
\label{sequentially:complete:minimal:abelian:groups:INFPROD}
A sequentially complete minimal abelian group without infinite products of non-trivial subgroups is compact.
\end{theorem}

 \begin{proof}
Let $H$ be a sequentially complete minimal abelian group without infinite products of non-trivial subgroups. 
Since $H$ is minimal, it is precompact, and so its completion $G$ is a compact abelian group. Assume $G$ contains
an infinite product $ A = \prod_nA_n$ of non-trivial subgroups. As in the proof of Theorem \ref{compact:abelian:groups:without:infinite:products}, 
we may assume that each $A_n$ is closed. By Theorem \ref{minimality-criterion}, each subgroup $B_n= A_n \cap G$ of $H$ is non-trivial. 
Moreover, the direct sum $\bigoplus_n B_n\subseteq H$ is a dense subgroup of the subgroup $B =\prod_nB_n$ of $A$. Since $B$ is metrizable and since
$H$ is sequentially complete, we conclude that $H$ contains the infinite product $B$, a contradiction. Therefore,  $G$ contains no infinite products of non-trivial subgroups
By Corollary \ref{no:infinite:products>metrizable}, $G$ is metrizable. Hence the sequential completeness of $H$ yields $H = G$. 
\end{proof}

\section{The TAP-topology of a compact abelian group}

The family $\mathfrak J_G$ of closed subgroups of a compact abelian group $G$ that contains no infinite products, defined in Remark \ref{The:filter}, can be 
enlarged and considered in arbitrary compact abelian groups, as we show in the next definition. 
 
\begin{definition} For an arbitrary compact abelian group $G$ one can introduce he family $\mathfrak F_G$ of all subgroups $N$ of $G$ with this property
$G/N$ is TAP.  Since the quotient of a compact TAP group is TAP and finite products of TAP groups are TAP, one can easily see that the following properties hold:
\begin{itemize}
\item[(i)]  if $N_1, N_2 \in{\mathfrak{F}}_G$, then $N_1\cap N_2 \in \mathfrak{F}_G$;
\item[(ii)]  if $N \in{\mathfrak{F}}_G$ and $N_1$ is a closed subgroup of $G$ containing $N$, then also $N_1 \in \mathfrak{F}_G$.
 \end{itemize}
In other words, ${\mathfrak{F}}_G$ is a filter-base consisting of closed subgroups of $G$.
Therefore, it gives rise to a linear group topology $\mathcal T_{_{TAP}}^G$ on $G$ that we shall call the {\em TAP-topology} of $G$.
 \end{definition}

  Let us start with two examples.

\begin{itemize}
\item[(E$_1$)] For $G=\T^\omega$ the topology $\mathcal T_{_{TAP}}^G$ is complete, non-discrete and non-locally compact ($\mathcal T_{_{TAP}}^G$ coincides with the Tichonov product of the discrete groups $\T$). The non-local compactness of this topology is explained by Theorem \ref{dimension:TAPtop}(ii) below. 
\item[(E$_2$)] For $G = \widehat{\Q} $  the TAP-topology is locally compact, non-compact. This can be easily proved using the fact that $G$ has a closed subgroup 
$H\cong \prod_p \Z_p$ such that $G/H\cong \T$ (\cite[\S 3.5]{DPS}). This will immediately follows from the more general Theorem \ref{dimension:TAPtop}(ii) proved below. 
  \end{itemize}

 One can prove the following basic properties of $\mathcal T_{_{TAP}}^G$.

\begin{theorem}\label{basic:TAPtop} Let $(G,\tau)$ be a compact abelian group. 
\begin{itemize}
\item[(a)] $\mathcal T_{_{TAP}}^G$ is discrete if and only if $G$ is TAP.
\item[(ii)]  the TAP-topology $\mathcal T_{_{TAP}}^G$ is {\em finer} than the original topology $\tau$. 
\item[(iii)] if $N$ is a $\mathcal T_{_{TAP}}^G$-open $\tau$-closed subgroup of $G$, then
$\mathcal T_{_{TAP}}^N=\mathcal T_{_{TAP}}^G\restriction _N$.
 \end{itemize}
\end{theorem}
 
\begin{proof} (i) If  $\mathcal T_{_{TAP}}^G$ is discrete, then $G \cong G/\{0\}$ is TAP. If $G$ is TAP, then $\{0\}\in {\mathfrak{F}}_G$, so $\mathcal T_{_{TAP}}^G$ is discrete. 

(ii)  According to well-known properties of the compact abelian groups (\cite[\S 3]{DPS}), for every open neighborhood $U$ of 0 in $G$ there exists a closed subgroup $N$ of $G$ such that $G/N$ is an elementary compact abelian group, i.e., a compact Lie group. Then $N \in  {\mathfrak{F}}_G$. This proves that $\tau \subseteq \mathcal T_{_{TAP}}^G$. 

(iii) Let $N_1$ be a $\mathcal T_{_{TAP}}^N$-open $\tau$-closed subgroup of $N$. Then $N/N_1$ is a Lie group. Since $G/N$ is a Lie group as well, 
we conclude that also $G/N_1$ is a Lie group. Therefore, $N_1$ is  also $\mathcal T_{_{TAP}}^G$-open. 
On the other hand, the inclusion $\mathcal T_{_{TAP}}^N\supseteq \mathcal T_{_{TAP}}^G\restriction _N$ (i.e., the continuity of the inclusion $N \hookrightarrow G$
w.r.t. the TAP topologies) will be proved below under more general hypotheses). 
  \end{proof}

\begin{theorem}\label{dimension:TAPtop} Let $(G,\tau)$ be a compact abelian group. 
\begin{itemize}
\item[(i)] $\mathcal T_{_{TAP}}^G$ coincides with $\tau$ iff  $G$ is totally disconnected.
\item[(ii)] $\mathcal T_{_{TAP}}^G$ is locally compact if and only if $G$ is finite-dimensional.
  \end{itemize}
\end{theorem}

\begin{proof}
(i) Assume $G$ is totally disconnected. Then every quotient $G/N$ of $G$ is totally disconnected as well. Therefore, $G/N$ is a Lie group if and only if $G/N$ is finite. Since
$G/N$ is TAP iff $G/N$ is a Lie group (Corollary  \ref{compact:abelian:group:are:Lie}), we conclude that the quotient $G/N$ for some closed subgroup $B$ of $G$ is TAP iff
$G/N$ is finite, i.e., iff $N$ is an open subgroup of $G$.  This proves the inclusion $ \mathcal T_{_{TAP}}^G \subseteq  \tau $. Now we can conclude that $\mathcal T_{_{TAP}}^G=\tau$ with item (ii) of Theorem \ref{basic:TAPtop}. Now assume that $G$ is not totally disconnected. Then there exists a surjective continuous character $\chi : G \to \T$, hence 
$N = \ker \chi$ is a $\mathcal T_{_{TAP}}^G$-open subgroup of $G$ that cannot be $\tau$-open, since it has infinite index. Therefore, $\mathcal T_{_{TAP}}^G\ne \tau$. 

(ii) Assume that $G$ is finite-dimensional and let $d = \dim G$. Then there exits a  continuous surjective homomorphism $f : G \to \T^d$. As in the proof of Theorem \ref{compact:abelian:groups:without:infinite:products} we conclude that $N= \ker f$ is totally disconnected. Since $G/N \cong \T^d$ is TAP, the subgroup $N$ is $\mathcal T_{_{TAP}}^G$-open. Now suppose that $N_1$ is another $\mathcal T_{_{TAP}}^G$-open $\tau$-closed subgroup of $G$. Then $G/N_1$ is a TAP group, so by Corollary  \ref{compact:abelian:group:are:Lie}, $G/N_1$ is a Lie group. Hence the quotient group $N/N\cap N_1$, being isomorphic to a closed subgroup of 
$G/N_1$, is a Lie group too. On the other hand, $N/N\cap N_1$ is totally disconnected, as a quotient of the totally disconnected compact group $N$. Hence
$N/N\cap N_1$ is finite. This proves that  $N\cap N_1$ is an open subgroup of $N$, provided with the induced by $\tau$ topology. Therefore, the induced by 
$\mathcal T_{_{TAP}}^G$ topology on $N$ coincides with the compact topology $\tau\restriction _N$ of $N$. This proves that $\mathcal T_{_{TAP}}^G$ is locally compact. 

Now assume that $\mathcal T_{_{TAP}}^G$ is locally compact. Then there exists a closed subgroup $N$ of $G$ such that $G/N$ is TAP (so  a Lie group, by Corollary  \ref{compact:abelian:group:are:Lie}) and $(N, \mathcal T_{_{TAP}}^G\restriction _N)$ is compact.  By item (iii) of Theorem \ref{basic:TAPtop} we can claim that $(N, \mathcal T_{_{TAP}}^N)$ is compact. Now item (ii) of Theorem \ref{basic:TAPtop} implies that $\mathcal T_{_{TAP}}^N= \tau\restriction_N$, as the former topology is compact.  Now item (i) yields that the subgroup $N$ is totally disconnected (in the induced by $\tau$ topology). Hence $\dim H =0$. Therefore, $\dim G = \dim G/N < \infty$, as $G/N$ is a Lie group. 
\end{proof}

In item (ii) of the above theorem we characterized the groups with locally compact TAP-topology. In both examples (E$_1$) and (E$_2$) the TAP-topology is complete. 
This justifies the following 

\begin{theorem}
The TAP-topology is complete for every compact abelian group $G$. 
\end{theorem}

\begin{proof}
We shall first consider the case when $G = \T^\kappa$ is a power of $\T$. In analogy with (E$_1$), one can easily see that the TAP-topology of $G$ 
coincides with the Tichonov product of the discrete groups $\T$, so it is complete. 

In the general case, $G$ is a closed subgroup of a power $ \T^\kappa$, so that it would be sufficient to extend item (iii) of Theorem \ref{basic:TAPtop} to arbitrary closed
subgroups, i.e., if $G$ is a closed subgroup of a compact group $N$. then $\mathcal T_{_{TAP}}^G=\mathcal T_{_{TAP}}^N\restriction _G$. In view of Theorem \ref{basic:TAPtop}(iii), it suffices to prove  only that if $L$ is a $\mathcal T_{_{TAP}}^G$-open subgroup of $G$, then there exists a $\mathcal T_{_{TAP}}^N$-open subgroup $H$ of $N$, such that
$L= H\cap G$. Since $G/L$ is a Lie group. we can find finitely many continuous characters $\chi_k: G\to \T$, $i=1,2,\ldots, n$, such that $L = \bigcap_i \ker \chi_i$. 
Let $\xi_i: N \to \T$ be a continuous character extending $\chi_i$, for  $i=1,2,\ldots, n$. Let $H =   \bigcap_i \ker \xi_i$. Then $N/H$ is isomorphic to a closed subgroup of $\T^k$, so
it is a Lie group. Therefore, $H$ is a $\mathcal T_{_{TAP}}^N$-open subgroup of $N$. Since each $\xi_i$ extends $\chi_i$, one can easily check that $L = H\cap G$. 
\end{proof}

For a compact group $G$ consider the topological group $\mathfrak{F}_{_{TAP}}(G):=(G,\mathcal T_{TAP}^G) $. For a continuous homomorphism $f: G \to H$ of compact  abelian
groups let $\mathfrak{F}_{_{TAP}}(f) = f$ as a set-map.

\begin{theorem}
The assignment $G\mapsto \mathfrak{F}_{_{TAP}}(G)$, $f\mapsto f=\mathfrak{F}_{_{TAP}}(f)$ defines is a covariant functor. 
\end{theorem}  

\begin{proof}
We have to show that if $f: G \to H$ is continuous homomorphism of compact abelian groups, then 
the homomorphism $f: (G, \mathcal T_{_{TAP}}^G) \to (H, \mathcal T_{_{TAP}}^H)$ is continuous. Take a closed subgroup $N$ of $H$ such that $H/N$ is a TAP group. 
Then $H/N$ is a Lie group by Corollary  \ref{compact:abelian:group:are:Lie}. Let $N_1= f^{-1}(N)$. Then $G/N_1$ is isomorphic to a closed subgroup of $H/N$, hence 
$G/N_1$ is a Lie group. Therefore, $N_1$ is $\mathcal T_{_{TAP}}^G$-open. 
\end{proof}

\section{TAP property in minimal abelian groups}

Let us recall the definition of a minimal group.

\begin{deff}
Let $G$ be a group. A Hausdorff group topology $\tau$ on $G$ is \emph{minimal} if
for every continuous isomorphism $f:(G,\tau)\to H$, where $H$ is Hausdorff, $f$ is a topological isomorphism.
\end{deff}

A description of the dense minimal subgroups of compact groups was given in the ``minimality criterion'' in \cite{P,S} in terms of
essential subgroups; a subgroup $H$ of a topological group $G$ is \emph{essential} if $H$ nontrivially intersects every nontrivial closed normal subgroup of $G$ \cite{P,S}.

\begin{theorem}\emph{\cite{DPS,P,S}}\label{minimality-criterion}
A dense subgroup $H$ of a compact group $K$ is minimal if and only if $H$ is essential in $G$.
\end{theorem}

Theorem \ref{minimality-criterion}
allows us to partially invert the trivial implication: if $H$ is a (dense) subgroup of a NSS group, then $H$ is NSS as well.

\begin{proposition}
\label{NSS:in:minimal:groups:is:inherited:from:completion}
Let $H$ be a dense minimal subgroup of an abelian topological group $G$. Then the following are equivalent:
\begin{itemize}
\item[(a)] $G $ is NSS;
\item[(b)] $H$ is NSS.
\end{itemize}
\end{proposition}

\begin{proof} (a) $\to$ (b) is obvious.

(b) Assume that $H$ is NSS. Then this is witnessed by a neighborhood $U$ of 0 in $H$. We can assume, without loss of  generality, that $U$ is closed. Then there exists a closed neighborhood $W$ of 0 in $G$ such that $W \cap H = U$. Assume $N$ is a subgroup of $G$ contained in $W$. Then also the closure of $N$ is contained in $W$, so we may assume, without loss of  generality, that $N$ itself is closed. Then $N \cap H$ is a closed subgroup of $H$ contained in $ U = H \cap W$. Therefore, $N \cap H=0$. Now Theorem
\ref{minimality-criterion} allows us to conclude that $N=0$ as well. Hence $W$ witnesses the NSS property for $G$.
\end{proof}

A topological group $G$ is {\em totally minimal\/} if every quotient group of $G$ is minimal.

\begin{example}\label{example2}
\begin{itemize}
\item[(i)] {\em A totally minimal separable metric linear abelian TAP group need not be NSS\/}. Indeed, let $(p_n)_n$ be a sequence of pairwise disjoint prime numbers. Then
the group $G = \bigoplus_n \Z(p_n)$ equipped with the product topology is a TAP group by Theorem \ref{direct:sum:of:TAP:is:TAP}, although $G$
 is not an NSS group. It is known that $G$ is totally minimal \cite{DPS}.
\item[(ii)]
The group $G$ from Example \ref{question1} is known to be totally minimal. Thus, we also have an example of a {\em cyclic\/} group with the same properties as in item (i).
\end{itemize}
\end{example}

The next example shows that the counterpart of Proposition \ref{NSS:in:minimal:groups:is:inherited:from:completion} for the property TAP fails.
\begin{example}
The (totally) minimal group $G$ from Example \ref{example2}(i) is TAP, while its completion is an infinite product $\prod_n \Z(p_n)$ of nontrivial groups, so trivially fails to be TAP.
\end{example}

\begin{theorem}
\label{sequentially:complete:minimal:abelian:groups:are:Lie}
A sequentially complete minimal abelian TAP group is a compact Lie group.
\end{theorem}

\begin{proof}
Let $H$ be a sequentially complete minimal abelian TAP group. Since $H$ contains no infinite products of non-trivial subgroups, $H$ is compact
by Theorem \ref{sequentially:complete:minimal:abelian:groups:INFPROD} $H$ is compact. 
Now Theorem \ref{compact:abelian:group:are:Lie} applies. 
\end{proof}

\begin{corollary}
\label{countably:compact:minimal:groups}
For a  sequentially complete minimal abelian group $G$ the following are equivalent:
\begin{itemize}
\item[(a)] $G$ is TAP;
\item[(b)] $G$ is NSS;
\item[(c)] $G$ is a compact Lie group.
\end{itemize}
\end{corollary}

Example \ref{example2} shows that ``sequential completeness'' cannot be dropped in the above corollary, even in the presence of total minimality.

The implication TAP $\to$ NSS fails for both (totally) minimal (metrizable) abelian groups (Example \ref{example2}) and (consistently) for countably compact abelian groups (Example \ref{example3}(ii)). Our next corollary shows that combining these two properties allow us to prove the implication TAP $\to$ NSS.

\begin{corollary}\label{count:comp:min:abelian:group:NSS:iff:TAP}
 For a countably compact minimal abelian group G the following are equivalent:
\begin{itemize}
\item[(i)] G is TAP,
\item[(ii)] G is a compact Lie group,
\item[(iii)] G is NSS.
\end{itemize}
\end{corollary}
\begin{proof} Countably compact  groups are sequentially complete \cite{DT1,DT2}, so Corollary \ref{countably:compact:minimal:groups} applies.
\end{proof}

\begin{question}
Can ``abelian'' be dropped in Corollary \ref{count:comp:min:abelian:group:NSS:iff:TAP}?
\end{question}

We do not know whether  countable compactness in
Corollary \ref{count:comp:min:abelian:group:NSS:iff:TAP}
can be weakened to pseudocompactness:

\begin{question}\label{question5}
Does there exist an example of a pseudocompact minimal abelian TAP group that is not NSS?
\end{question}

Note that the only examples available of minimal TAP groups that fail to be NSS are metrizable (hence cannot be pseudocompact).

\section{An example}

Proposition \ref{non-TAP:subgroups:of:Z_p} may leave the impression that a metrizable non-TAP group $G$ must be in some sense close to being complete in case the
subgroup generated by the AP subset $A$ is dense in $G$. However, this fails to be true as the following easy example shows:

\begin{example}
Let $G_n$ be an infinite non-discrete monothetic metrizable group for every $n\in \N$. Denote by $C_n$ a dense proper cyclic subgroup  subgroup of $G_n$
and let $a_n$ be a generator of $C_n$.  Then the subset $A=\{a_n: n\in \N\}$ of the infinite product $C = \prod_n C_n$ is an $f_\omega$-productive
set. So $C$ is
a  metrizable  non-TAP group that is not complete and $A$ generates a dense subset of $C$.
\end{example}

After seeing this example, one may be tempted to modify the question by formulation the following conjecture: {\em If $A=(a_n)$ is an AP sequence in a (metric) group $G$ and $H $ is the smallest closed subgroup of $G$ containing $A$, then either $H$ is complete or $H$ contains an infinite product of nontrivial groups.\/} The main goal in this section is to produce a counter-example to this conjecture; see Proposition \ref{proposition:example}.

\begin{lemma}
\label{linear:algebra:excercise} If $k,l,m,n\in\N$, $k\not=0$ and $n\not=0$, then there exist $x,y\in\{1,2\}$ such that $kx+ly\not=0$ and $mx+ny\not=0$.
\end{lemma}

Call a family $\mathcal{S}$ of subsets of a given set $X$:
\begin{itemize}
\item[(i)]
{\em $2$-linked\/} if any two distinct members of $\mathcal{S}$ have an infinite intersection,
\item[(ii)]
{\em $3$-disjoint\/} if any three pairwise distinct members of
$\mathcal{S}$ have an empty intersection.
\end{itemize}

\begin{lemma}
There exists a $2$-linked $3$-disjoint
faithfully indexed
family
$\mathcal{S}=\{S_n:n\in\N\}$
of subsets of $\N$ satisfying $\bigcup\mathcal{S}=\N$.
\end{lemma}
\begin{proof}
Let $[\N]^2$ be the family of all two-element subsets of $\N$, and let $f:\N\to
[\N]^2$ be any surjection such that the set $\{k\in\N: f(k)=\{n,m\}\}$
is infinite for every $\{n,m\}\in[\N]^2$.
For $n\in\N$ define $S_n=\{k\in\N: n\in f(k)\}$.
A straightforward check that $\mathcal{S}=\{S_n:n\in\N\}$ has the required properties is left to the reader.
\end{proof}

Recall that a family $\mathcal{T}$ of subsets of $\N$ is called {\em independent\/} if the intersection $T_1\cap \ldots \cap T_k\cap
(\N\setminus T_1')\cap\dots \cap(\N\setminus T_l')$ is infinite for every pair of integers $k,l\in\N$ and each pairwise distinct
sequence $T_1,\ldots,T_k, T_1',\dots, T_l'$ of members of $\mathcal{T}$. It is well known  that there exists an independent
family of size continuum on $\N$, so we can choose a ``countable piece'' $\mathcal{T}=\{T_{n}:n\in\N\}$ of that family.

Let $P$ be a countably infinite subset of the set $\mathbb{P}$ of prime numbers such that $2\not\in P$, and let
${P}=\{p_{i,j}:(i,j)\in\N^2\}$ be a faithful enumeration of the set ${P}$. Define
$$
K_P= \prod_{(i,j)\in\N^2} \Z(p_{i,j}).
$$
For $g\in K_P$ let $s(g)=\{(i,j)\in\N^2: g(i,j)\not=0\}$. For each $n\in\N$ define $a_n\in K$ and by
\begin{equation}
\label{eq:defining:a_n} a_n(i,j) = \left\{
\begin{array}{ll}
0 & \mbox{if $i\in \N\setminus S_n$}\\
1 & \mbox{if $i\in S_n$ and $j\in T_n$}\\
2 & \mbox{if $i\in S_n$ and $j\in \N\setminus T_n$}
\end{array}
\right. \hskip40pt \mbox{for } (i,j)\in\N^2.
\end{equation}

\begin{claim}
\label{construction:of:a_n}
The sequence $A=\{a_n:n\in\N\}$ satisfies the following conditions:
\begin{itemize}
\item[(i)]
$\{s(a_n):n\in\N\}$ is a $2$-linked $3$-disjoint family of subsets
of $\N^2$;
\item[(ii)]
$\bigcup_{n\in\N}s(a_n)=\N^2$;
\item[(iii)]
if $q,r\in\N$, $q\not=r$, $k,l,m,n\in\Z$ and $k\neq 0\neq n$, then the set
$$
E(k,l,m,n,q,r)=s(ka_q+la_r)\cap s(ma_q+na_r)\cap s(a_q)\cap s(a_r)
$$
is infinite.
\end{itemize}
\end{claim}
\begin{proof}
From \eqref{eq:defining:a_n} it follows that $s(a_n)=S_n\times \N$ for every $n\in\N$. Now (i) and (ii) follow easily from our choice
of the family  $\mathcal{S}$. Let us check that (iii) holds as well. Since $\mathcal{S}$ is $2$-linked, the set $S=S_q\cap S_r$ is
infinite. In particular, we can choose $i\in S$. Let $x$ and $y$ be as in the conclusion of Lemma \ref{linear:algebra:excercise}. Let
\begin{equation}
\label{definition:of:F_pq} R_p= \left\{
\begin{array}{ll}
T_p & \mbox{if $x=1$}\\
\N\setminus T_p & \mbox{if $x=2$}
\end{array}
\right. \ \mbox{ and }\ R_q= \left\{
\begin{array}{ll}
T_q & \mbox{if $y=1$}\\
\N\setminus T_q & \mbox{if $y=2$}.
\end{array}
\right.
\end{equation}

Since $\mathcal{T}$ is independent and $T_p,T_q\in\mathcal{T}$, from
\eqref{definition:of:F_pq} it follows that the set
\begin{equation}
\label{definition:of:T} T=R_p\cap R_q
\end{equation}
is infinite. Since the set
\begin{equation}
\label{definition:of:J} J=\{j\in \N: p_{i,j}\le \max(|kx+ly|,
|mx+ny|)\}
\end{equation}
 is finite,
$T'=T\setminus J$ is infinite. From \eqref{eq:defining:a_n},
\eqref{definition:of:T} and \eqref{definition:of:J} we conclude that
$$
a_q(i,j)=x\not=0, a_r(i,j)=y\not=0, (ka_q+la_r)(i,j)=kx+ly\not=0
\mbox{ and } (ma_q+na_r)(i,j)=mx+ny\not=0
$$
for each $j\in T'$. This proves that $\{i\}\times T'\subseteq
E(k,l,m,n,q,r)$.
\end{proof}

\begin{claim}\label{chinese:remainder:theorem}
If $g\in K_P$ and $F$ is a finite subset of $s(g)$, then the projection $\pi_F:\langle g\rangle\to\prod_{(i,j)\in F}\Z(p_{i,j})$ defined by
$\pi_F(f)=f\restriction_F$ for every $f\in\langle g\rangle$, is surjective.
\end{claim}

\begin{proof}
This immediately follows from the Chinese remainder theorem.
\end{proof}

\begin{claim}\label{no:products:inside}
Given $g,g'\in K_P\setminus\{0\}$, consider the topological group $N=\langle
g\rangle\times\langle g'\rangle$, where $\langle g\rangle$ and
$\langle g'\rangle$ are taken with the subspace topology induced
from $K_P$, and let
$\phi:N\to K_P$ be the homomorphism
defined by  $\phi(mg,ng')=mg+ng'$ for $m,n\in\Z$.
Then $\phi$ is a topological isomorphism
between $N$ and $\phi(N)=\langle\{g,g'\}\rangle$ if and only if $s(g)\cap s(g')=\emptyset$.
\end{claim}

\begin{proof}
The ``if'' part is obvious. To prove the ``only if'' part, assume that $(i_0,j_0)\in s(g)\cap s(g')$.
If $\phi$ is not a monomorphism, the proof is complete. So we shall assume that $\phi$ is a monomorphism.
Take the basic open neighborhood $U=\{g\in
K_P:g(i_0,j_0)=0\}$ of $0$ in $K_P$. Then
$V=(U\cap\langle
g\rangle)\times\langle g'\rangle$ is an open subset of $N$, so it suffices to
show that the set  $\phi(V)$
in not open in $\phi(H)$. Assume the contrary. Since $0\in\phi(V)$,
there exists a finite set $F\subseteq \N^2$ such that
$W\cap \phi(H)\subseteq \phi(V)$, where
$W=\{g\in K_P: g(i,j)=0$ for all $(i,j)\in F\}$.
Without loss of generality, we may assume that $(i_0,j_0)\in F$.
Since $(i_0,j_0)\in s(g)\cap s(g')$, applying Claim
\ref{chinese:remainder:theorem} we can choose
$z,z'\in\Z$ such that
\begin{equation}
\label{eq:14}
zg(i_0,j_0)=1
\mbox{ and }
zg(i,j)=0
\mbox{ for all }
(i,j)\in (F\cap s(g))\setminus \{(i_0,j_0)\},
\end{equation}
\begin{equation}
\label{eq:15}
z'g'(i_0,j_0)=p_{i_0,j_0}-1
\mbox{ and }
z'g'(i,j)=0
\mbox{ for all }
(i,j)\in (F\cap s(g'))\setminus \{(i_0,j_0)\}.
\end{equation}
From \eqref{eq:14} and \eqref{eq:15} it follows that
$(zg+z'g')(i,j)=0$ for all $(i,j)\in F$;
that is,
$\phi(zg,z'g')=zg+z'g'\in W\cap\phi(N)\subseteq \phi(V)$.
Since $\phi$ is assumed to be a monomorphism,
we conclude that
$(zg,z'g')\in V$.
In particular, $zg\in U$, which yields $zg(i_0,j_0)=0$.
This contradicts \eqref{eq:14}.
\end{proof}

\begin{claim}
$g_z=\lim_{k\to\infty}\sum_{n=0}^k z(n) a_n\in K_P$ for every $z\in
\Z^{\N}$.
\end{claim}
\begin{proof}
This follows from Claim \ref{construction:of:a_n}(i) and Lemma \ref{char:of:ap:in:K}.
\end{proof}

The last claim allows us to define $H_P=\left\{g_z: z\in\Z^{\N}\right\}\subseteq K_P$. Let $z,z'\in \Z^{\N}$. Since $K_P$ is abelian, we have
$$
\sum_{n=0}^k z(n)a_n+\sum_{n=0}^k z'(n)a_n=\sum_{n=0}^k (z(n)+z'(n))a_n$$ for every $k\in\N$. Thus, using the continuity of the group operation we get $g_z+g_{z'}=g_{z+z'}$. Hence, $H_P$ is a subgroup of $K_P$.

\begin{claim}
\label{any:two:supports:intersect} If $z,z'\in\Z^{\N}$, $g_z\not=0$ and $g_{z'}\not=0$, then $s(g_z)\cap s(g_{z'})\not=\emptyset$.
\end{claim}
\begin{proof}
Since $g_z\neq 0\neq g_{z'}$, we have $\{n\in\N:z(n)\neq 0\}\not=\emptyset$ and $\{n\in\N:z'(n)\neq 0\}\not=\emptyset$. We consider two cases.

{\sl Case 1\/}. {\it $\{n\in\N:z(n)\neq 0\}=\{n\in\N:z'(n)\neq 0\}=\{q\}$ for some $q\in\N$.}
In this case
$g_z=z(q)a_q$ and $g_{z'}=z'(q)a_q$.
Since the family $\{s(a_n):n\in\N\}$ is $2$-linked by
Claim \ref{construction:of:a_n}(i), $s(a_q)$ must be infinite.
Note that $s(g_z)=s(z(q) a_q)=s(a_q)\setminus F_z$
and
$s(g_{z'})=s(z'(q) a_q)=s(a_q)\setminus F_{z'}$,
where $F_z=\{(i,j)\in\N^2: p_{i,j}$ divides $z(q)\}$
and $F_{z'}=\{(i,j)\in\N^2: p_{i,j}$ divides $z'(q)\}$
are finite sets, so the set
$s(g_z)\cap s(g_{z'})=s(a_q)\setminus(F_z\cup F_{z'})$
must be infinite.

{\sl Case 2\/}. {\it There exist distinct $q\in\N$ and $r\in\N$ such that $z(q)\not=0$ and $z'(r)\not=0$.\/} Define $k=z(q)$, $l=z(r)$,
$m=z'(q)$ and $n=z'(r)$. By Claim \ref{construction:of:a_n}(iii), we can choose $(i,j)\in E(k,l,m,n,q,r)\subseteq s(a_q)\cap s(a_r)$.
Since the family $\{s(a_n):n\in\N\}$ is $3$-disjoint by Claim \ref{construction:of:a_n}(iii), it follows that $a_p(i,j)=0$ for all $p\in\N\setminus\{q,r\}$. Therefore,
\begin{equation}
\label{eq:16}
g_z(i,j)=\sum_{n=0}^\infty z(n) a_n(i,j)=z(q)a_q(i,j)+z(r)a_r(i,j)
=
(ka_q+la_r)(i,j)
\end{equation}
and
\begin{equation}
\label{eq:17}
g_{z'}(i,j)=\sum_{n=0}^\infty z'(n) a_n(i,j)=z'(q)a_q(i,j)+z'(r)a_r(i,j)
=
(ma_q+na_r)(i,j).
\end{equation}
Since
$(i,j)\in E(k,l,m,n,q,r)\subseteq s(ka_q+la_r)\cap s(ma_q+na_r)$,
from
\eqref{eq:16} and \eqref{eq:17} we conclude that $g_z(i,j)\not =0$ and $g_{z'}(i,j)\not=0$. In particular, $(i,j)\in s(g_z)\cap s(g_{z'})\not=\emptyset$.
\end{proof}

\begin{claim}
\label{denseness:claim}
$H_P$ is dense in $K_P$.
\end{claim}
\begin{proof}
Using Claim \ref{chinese:remainder:theorem}, we can easily derive the following stronger statement. If $K\subset K_P$ and $F$ is a
finite subset of $\bigcup_{g\in K}s(g)$, then the projection $\pi_F:\langle K\rangle\to\prod_{(i,j)\in F}\Z(p_{i,j})$ defined by
$\pi_F(f)=f\restriction_F$ for every $f\in\langle K\rangle$, is surjective. Now the conclusion follows from Claim
\ref{construction:of:a_n}(ii).
\end{proof}

\begin{proposition}
\label{proposition:example}
The group $H_P$ has the following properties:
\begin{itemize}
\item[(i)] $H_P$ is the smallest subgroup $H$ of $K_P$ such that the set $\{a_n:n\in\N\}$ is \ap\ in $H$; in particular, $H_P$ is not TAP.
\item[(ii)] $H_P$ does not contain a subgroup topologically isomorphic to a product of two nontrivial topological groups.
\item[(iii)] $H_P$ is not complete.
\end{itemize}
\end{proposition}
\begin{proof}
(i) It follows immediately from the definition of $H_P$ that $H_P$ is the smallest subgroup $H$ of $K_P$ such that $\{a_n:n\in\N\}$ is an \ap\ sequence in $H$. Since the family $\{s(a_n):n\in\N\}$ is 3-disjoint by Claim \ref{construction:of:a_n}(i), and $K_P$ is abelian, it follows easily that $\{a_n:n\in\N\}$ is an \ap\ set in $H_P$.

(ii) Assume that $M$ is a subgroup of $H_P$ that is topologically isomorphic to a product of nontrivial topological groups $L$ and
$L'$, and let $\theta:L\times L'\to M$ be a topological isomorphism. Choose $x\in L\setminus\{0\}$ and $x'\in L'\setminus \{0\}$, and let
$g=\theta(x,0)$ and $g'=\theta(0,x')$. Then the map $\phi$ from Claim \ref{no:products:inside} becomes a topological isomorphism, so
we must have $s(g)\cap s(g')=\emptyset$ by Claim \ref{no:products:inside}. This contradicts Claim \ref{any:two:supports:intersect}.

(iii) From (ii) it follows that $H_P\not=K_P$. Since $H_P$ is a proper dense subgroup of the compact group $K_P$ by Claim \ref{denseness:claim}, we conclude that $H_P$ cannot be complete.
\end{proof}

\section{TAP topologizations}
\begin{theorem}\label{opens:subgroups<c:implies:TAP} If topological group $G$ has an open normal subgroup $N$ of size $< \cont$, then $G$ is TAP.
\end{theorem}

\begin{proof} Follows from Corollary \ref{opens:subgroups} and the fact that $N$ is TAP according to Corollary \ref{|G|<cont:is:TAP}.
\end{proof}

As we show in the next example, it easily follows from the above corollary that every infinite abelian group admits a non-discrete TAP group topology.

\begin{example} Let $G$ be an infinite abelian group. Take a countably infinite subgroup $N$ of $G$, a non-discrete group topology $\tau$ on $H$ and extend $\tau$ to a group topology of $G$ declaring $H$ an open topological subgroup of $G$. The above corollary ensures that $G$ is TAP. By choosing $\tau$ metrizable, one can get
a metrizable non-discrete group topology on $G$. Let us note that the metrizable non-discrete TAP topologies obtained in this way are not separable if the group $G$ is not countable (cf. Question \ref{separable:TAP:topology}).
\end{example}

Now we show that the TAP group topology can be chosen precompact.

\begin{corollary}
\label{every:free:abelian:group:has:TAP:topology}
Every abelian group $G$ admits a (Hausdorff) precompact TAP group topology.
\end{corollary}

\begin{proof} Since the group $\T$ is divisible, with torsion part $\Q/\Z$, it follows from the structure theory of divisible groups that we can find a subgroup $H$ 
of $\T$ isomorphic to
 $\Q/\Z \oplus \Q $.  By Corollary \ref{|G|<cont:is:TAP}, $H$ is TAP (in the induced topology). Then,
for every cardinal $\sigma$, the subgroup $H_\sigma=\bigoplus_\sigma H$ of $\T^\sigma$
 is TAP by Theorem \ref{direct:sum:of:TAP:is:TAP}. Since our group $G$ is algebraically isomorphic to a subgroup of
$H_\sigma$ for an appropriate $\sigma$ (e.g., $\sigma = |G|$), this algebraic isomorphism induces on $G$ a precompact and TAP group
topology.
\end{proof}

It is known that every abelian group of size $\leq 2^\cont$ admits a
separable (precompact) group topology \cite{DS_monothetic}. This
motivates  the following question:

\begin{question}
\label{separable:TAP:topology} Does every abelian group of size
$\leq 2^\cont$ admit a non-discrete separable (precompact) TAP group topology?
\end{question}

It is known that every free group admits precompact (hence, non-discrete) group topologies. This justifies the following:

\begin{question}
\label{free:group:have:TAP:topology} Does every free group admit a non-discrete (precompact) TAP group topology?
\end{question}

Since NSS precompact groups have size at most $\cont$, the following theorem provides us with many TAP groups that are not NSS.

\begin{theorem}
Every abelian group admits a  (Hausdorff) precompact  non-$f_1$-Cauchy productive (hence, TAP) group topology.
\end{theorem}
\begin{proof} If $G$ is an infinite abelian group and $G^\sharp$ denotes $G$ equipped with its Bohr topology 
(=maximum precompact topology), then $G^\sharp$ has no nontrivial convergent sequences according to  a well known fact 
proved by Reid \cite{Re}, so by Lemma \ref{AP:is:null:sequence} $G^\sharp$ contains no $f_1$-Cauchy productive sequences.
\end{proof}

\medskip
\noindent
{\bf Acknowledgement:} We would like to thank cordinally Vaja Tarieladze for fruitful exchange of ideas that inspired
us to prove Theorem \ref{NSS:is:NACP} and Corollary \ref{weil:NSS:iff:TAP}.

\end{document}